\documentclass[acmsmall]{acmart}
\usepackage{graphicx} 
\usepackage{csquotes}
\usepackage{tikz,tikz-cd}
\usetikzlibrary{fit} 
\usetikzlibrary{graphs,decorations.pathmorphing,decorations.markings, arrows,arrows.meta,positioning}
\usepackage{mathpartir}
\usepackage{xcolor}
\definecolor{blue1}{RGB}{0, 0, 200}
\definecolor{red1}{RGB}{160, 0, 0}
\usepackage{array}
\usepackage{multicat}
\usepackage{booktabs}
\usepackage{accents}
\usepackage{adjustbox}
\usepackage{prooftree}
\usepackage{orcidlink}
\usepackage{mdframed}

\newcommand{\sppan}[5]{%
\begin{tikzcd}[ampersand replacement = \&, row sep = small]
\& {#3} \ar[dl, "#4"'] \ar[dr, "#5"] \& \\
{#1} \& \& {#2}
\end{tikzcd}
}
\newcommand{\etimes}{\,\times\,}
\newcommand{\erto}{\,\rightarrow\,}

\newcommand{\dynsup}{\mathbf{DynSup}}
\newcommand{\seman}[1]{ [\![ {#1} ]\!]}
\newcommand{\leftskewify}{\iota_{\mathsf{Left}}}
\newcommand{\biskewify}{\iota_{\mathsf{LtoBi}}}

\newcommand{\cats}{\mathcal{S}}
\newcommand{\retseq}[1]{\mathsf{RetSeq}({#1})}
\newcommand{\vdashv}{\vdash^{\mathsf{v}}}
\newcommand{\vdashc}{\vdash^{\mathsf{c}}}
\newcommand{\monmult}[1]{\mathbf{Mult}({#1})}
\newcommand{\ttproduce}{\ensuremath{\mathtt{return}} \ }
\newcommand{\ttskip}{\ensuremath{\mathtt{skip}} \ }
\newcommand{\eretof}[2]{\mathsf{ret}^{#1}_{#2}}
\newcommand{\etoof}[3]{\mathsf{to}^{#1}_{{#2},{#3}}}

\newcommand{\af}{\ }
\newcommand{\nsp}{\hspace{4.5pt}}
\newcommand{\retof}[2]{\mathsf{ret}}
\newcommand{\shortretof}[2]{\mathsf{retpi}^{#1}_{#2}}
\newcommand{\shortretoof}[2]{\mathsf{retpi}}
\newcommand{\toof}[3]{\mathsf{to}}
\newcommand{\piof}[2]{\pi_{{#1},{#2}}}
\newcommand{\pipof}[2]{\pi'_{{#1},{#2}}}
\newcommand{\tuple}[1]{\langle {#1} \rangle}
\newcommand{\ttto}{\  \mathtt{to} \ }
\newcommand{\ttin}{. \ }
\newcommand{\seqof}[1]{\mathbf{Seq}({#1})}
\newcommand{\ulb}{\underline{B}}
\newcommand{\ulc}{\underline{C}}

\newcommand{\lista}{\ora{a}}
\newcommand{\listb}{\ora{b}}
\newcommand{\listc}{\ora{c}}
\newcommand{\listd}{\ora{d}}
\newcommand{\liste}{\ora{e}}
\newcommand{\ora}[1]{\vec{#1}}
\newcommand{\orap}[1]{\accentset{\twoheadrightarrow}{#1}}
\newcommand{\listap}{\orap{a}}
\newcommand{\listbp}{\orap{b}}
\newcommand{\listcp}{\orap{c}}

\newcommand{\nats}{\mathbb{N}}
\newcommand{\catc}{\mathcal{C}}
\newcommand{\catn}{\mathcal{N}}
\newcommand{\catd}{\mathcal{D}}

\newcommand{\cato}{\mathcal{O}}
\newcommand{\catr}{\mathcal{R}}
\newcommand{\catt}{\mathcal{T}}

\newcommand{\pset}{\mathcal{P}}

\newcommand{\setbr}[1]{\{{#1} \}}
\newcommand{\op}{\mathsf{op}}
\newcommand{\forbim}[1]{\mathsf{Forward}_{#1}}
\newcommand{\revbim}[1]{\mathsf{Backward}_{#1}}
\newcommand{\eqdef}{\stackrel{\mbox{\rm {\tiny def}}}{=}}

\newcommand{\twotext}[2]{\begin{array}{l} \text{#1} \\ \text{#2} \end{array}}
\newcommand{\setval}{\mathbf{SetMat}}

\newcommand{\Set}{\mathbf{Set}}
\newcommand{\Span}{\mathbf{Span}}
\newcommand{\Cat}{\mathbf{Cat}}
\newcommand{\smin}{\!\in\!}

\newcommand{\ccolc}{:}
\newcommand{\obs}[1]{|{#1}|}
\newcommand{\unitof}[1]{\eta_{#1}}
\newcommand{\exof}[1]{{#1}^{*}}
\newcommand{\id}{\mathsf{id}}

\newcommand{\terminal}{\mathbf{1}}

\NewDocumentCommand{\nti}{s O{1} m O{} O{A} m O{} O{A} m}{%
 \IfBooleanTF{#1}%
  {%
   \cmorphism[#2]{#3}[#4][#5]{#6}[#7][#8]{#9}[\invtrapeze]%
  }%
  {%
   \morphism[#2]{#3}[#4][#5]{#6}[#7][#8]{#9}[\invtrapeze]%
  }%
}

\NewDocumentCommand{\nto}{s O{1} m m m m m}{
 \IfBooleanTF{#1}%
  {%
   \cmorphism[#2]{#3}[#4~][#4~]{#5}[~#6][~#6]{#7}[\invtrapeze_{c}]%
  }%
  {%
   \morphism[#2]{#3}[#4~][#4~]{#5}[~#6][~#6]{#7}[\invtrapeze_{c}]%
  }%
}
\NewDocumentCommand{\ntile}{s O{1} m O{} O{A} m O{} O{A} m}{
 \IfBooleanTF{#1}%
  {%
    \IfBlankTF{#4}%
    {%
        \cmorphism[#2]{#3}{\lti{#6}~}[#7][#8]{#9}[\invtrapeze^{l}]%
    }
    {%
        \cmorphism[#2]{#3}[#4][#5]{#6}[#7][#8]{#9}[\invtrapeze^{l}]%
    }
  }%
  {%
    \IfBlankTF{#4}%
    {%
        \morphism[#2]{#3}{\lti{#6}~}[#7][#8]{\lti{#9}}[\invtrapeze^{l}]%
    }
    {%
        \morphism[#2]{#3}[#4][#5]{#6}[#7][#8]{\lti{#9}}[\invtrapeze^{l}]%
    }
  }%
}
\NewDocumentCommand{\ntilesans}{O{1} m O{} m m}{
  \IfBlankTF{#3}%
    {%
        \morphism[#1]{#2}{\lti{#4}~}{#5}[\invtrapeze^{l}]%
    }
    {%
        \morphism[#1]{#2}[#3]{#4}{#5}[\invtrapeze^{l}]%
    }
}

\NewDocumentCommand{\ntiri}{s O{1} m O{} O{A} m O{} O{A} m}{
 \IfBooleanTF{#1}%
  {%
    \IfBlankTF{#7}%
    {%
        \cmorphism[#2]{#3}[#4][#5]{~\rti{#6}}{#9}[\invtrapeze^{r}]%
    }
    {%
        \cmorphism[#2]{#3}[#4][#5]{#6}[#7][#8]{#9}[\invtrapeze^{r}]%
    }
  }%
  {%
    \IfBlankTF{#7}%
    {%
        \morphism[#2]{#3}[#4][#5]{~\rti{#6}}{\rti{#9}}[\invtrapeze^{r}]%
    }
    {%
        \morphism[#2]{#3}[#4][#5]{#6}[#7][#8]{\rti{#9}}[\invtrapeze^{r}]%
    }
  }%
}
\NewDocumentCommand{\ntirisans}{O{1} m m O{} m}{%
  \IfBlankTF{#4}%
    {%
        \morphism[#1]{#2}{~\rti{#3}}{#5}[\invtrapeze^{r}]%
    }
    {%
        \morphism[#1]{#2}{#3}[#4]{#5}[\invtrapeze^{r}]%
    }
}

\NewDocumentCommand{\ntilr}{s O{1} m O{} O{A} m O{} O{A} m}{%
 \IfBooleanTF{#1}%
  {%
    \IfBlankTF{#4}%
    {%
        \IfBlankTF{#7}%
        {%
            \cmorphism[#2]{#3}{\lrti{#6}}{#9}[\invtrapeze^{lr}]%
        }
        {%
            \cmorphism[#2]{#3}{\lti{#6}}[#7][#8]{#9}[\invtrapeze^{lr}]%
        }
    }
    {%
        \IfBlankTF{#7}%
        {%
            \cmorphism[#2]{#3}[#4][#5]{\rti{#6}}{#9}[\invtrapeze^{lr}]%
        }
        {%
            \cmorphism[#2]{#3}[#4][#5]{#6}[#7][#8]{#9}[\invtrapeze^{lr}]%
        }
    }
  }%
  {%
    \IfBlankTF{#4}%
    {%
        \IfBlankTF{#7}%
        {%
            \morphism[#2]{#3}{\lrti{#6}}{#9}[\invtrapeze^{lr}]%
        }
        {%
            \morphism[#2]{#3}{\lti{#6}}[#7][#8]{#9}[\invtrapeze^{lr}]%
        }
    }
    {%
        \IfBlankTF{#7}%
        {%
            \morphism[#2]{#3}[#4][#5]{\rti{#6}}{#9}[\invtrapeze^{lr}]%
        }
        {%
            \morphism[#2]{#3}[#4][#5]{#6}[#7][#8]{#9}[\invtrapeze^{lr}]%
        }
    }
  }%
}
\NewDocumentCommand{\ntilrunary}{s O{1} m m m}{
\IfBooleanTF{#1}%
    {%
        \cunarymorphism[#2]{#3}{#4}[\lrti{x}]{#5}[\invtrapeze^{lr}]%
    }
    {%
        \unarymorphism[#2]{#3}{#4}[\lrti{x}]{#5}[\invtrapeze^{lr}]%
    }
}

\newcommand{\lloo}[1]{\color{purple}\acute{\color{black}#1}\color{black}}
\newcommand{\rloo}[1]{\color{purple}\grave{\color{black}#1}\color{black}}
\newcommand{\lrloo}[1]{\color{purple}\hat{\color{black}#1}\color{black}}
\newcommand{\lti}[1]{\smash{\lfloor}{#1}}
\newcommand{\rti}[1]{{#1}\rfloor}
\newcommand{\lrti}[1]{\lfloor{#1}\rfloor}
\newcommand{\longquad}{\quad\quad\quad\quad}
\newcommand{\ttx}{x}

\newcommand{\as}[2]{\frac{\mathrm{#1}}{\mathrm{#2}}}

\newcommand{\ch}[2]{\text{\small{\rm{#1}{#2}}}} 

\newcommand{\multicat}{\mathsf{Mult}}
\newcommand{\hiitem}[1]{\emph{#1:}}
\newcommand{\relarrow}{\, \to\hspace{-0.65em}\shortmid \hspace{0.5em}}
\newcommand{\monsof}[1]{\mathbf{Mon}({#1})}
\newcommand{\unbmon}[1]{\mathbf{UnbiasMon}({#1})}
\newcommand{\xn}{x^{(n)}}
\newcommand{\multof}[1]{m_{#1}}
\newcommand{\lmultof}[1]{m^{\mathsf{L}}_{#1}}
\newcommand{\tmultof}[1]{m^{\mathsf{T}}_{#1}}
\newcommand{\llmultof}[1]{m^{\mathsf{LL}}_{#1}}
\newcommand{\ltmultof}[1]{m^{\mathsf{LT}}_{#1}}
\newcommand{\tlmultof}[1]{m^{\mathsf{TL}}_{#1}}
\newcommand{\ttmultof}[1]{m^{\mathsf{TT}}_{#1}}

\newenvironment{spaceout}[1]{\begin{displaymath}\setlength{\extrarowheight}{3pt}\begin{array}{#1}}{\end{array}\setlength{\extrarowheight}{0pt}\end{displaymath} \noindent}
\newtheorem{theorem}{Theorem}
\newtheorem{definition}{Definition}
\newtheorem{example}{Example}

\newtheorem{proposition}{Proposition}
\theoremstyle{remark}
\newtheorem{remark}{Remark}

\setcopyright{cc}
\setcctype{by}
\acmDOI{10.1145/3776727}
\acmYear{2026}
\acmJournal{PACMPL}
\acmVolume{10}
\acmNumber{POPL}
\acmArticle{85}
\acmMonth{1}
\received{2025-07-10}
\received[accepted]{2025-11-06}

\begin{document}

\title{What Is a Monoid?}

\author{Paul Blain Levy}
\orcid{0000-0003-0864-1876}
\affiliation{%
  \institution{School of Computer Science, University of Birmingham}
  \city{Birmingham}
  \country{United Kingdom}}
\email{P.B.Levy@bham.ac.uk}

\author{Morgan Rogers}
\orcid{0000-0002-0277-8217}
\affiliation{%
  \institution{Laboratoire d'Informatique Paris Nord (LIPN), 
  UMR CNRS 7030 Universit\'{e} Sorbonne Paris Nord}
  \city{Villetaneuse 93430}
  \country{France}}
\email{rogers@lipn.univ-paris13.fr}


\begin{abstract}
In many situations one encounters an entity that resembles a monoid. It consists of a carrier and two operations that resemble a unit and a multiplication, subject to three equations that resemble associativity and left and right unital laws. The question then arises whether this entity is, in fact, a monoid in a suitable sense.

Category theorists have answered this question by providing a notion of monoid in a monoidal category, or more generally in a multicategory.  While these encompass many examples, there remain cases which do not fit into these frameworks, such as the notion of relative monad and the modelling of call-by-push-value sequencing.  In each of these examples, the leftmost and/or the rightmost factor of a multiplication or associativity law seems to be distinguished.  

To include such examples, we generalize the multicategorical framework in two stages.  

Firstly, we move to the framework of a left-skew multicategory (due to Bourke and Lack), which generalizes both multicategory and left-skew monoidal category.  The notion of monoid in this framework encompasses examples where only the leftmost factor is distinguished, such as the notion of relative monad.

Secondly, we consider monoids in the novel framework of a bi-skew multicategory.  This encompasses examples where both the leftmost and the rightmost factor are distinguished, such as the notion of a category on a span, and the modelling of call-by-push-value sequencing.

In the bi-skew framework (which is the most general), we give a coherence result saying that a monoid corresponds to an unbiased monoid, i.e. a map from the terminal bi-skew multicategory.
\end{abstract}

\keywords{monoid, skew multicategory, relative monad, call-by-push-value}

\ccsdesc[500]{Theory of computation, Categorical semantics}

\maketitle

\section{Introductions}
\subsection{Programming Language Introduction}
\label{sect:plintro}

In a simple programming language, given programs $M$ and $N$, we can form the program $M;N$ that executes $M$ and then $N$.  This operation on programs, known as \emph{sequencing}, satisfies three equations.  Firstly we have the \emph{associativity} law $(M;N);P = M;(N;P)$.  Writing $\ttskip$ for a program that does nothing, we also have the \emph{left identity} law $\ttskip; M = M$ and the \emph{right identity} law $M; \ttskip = M$.   Thus we see the algebraic structure of a \emph{monoid}.

However, sequencing is not always so simple.  Let us consider \emph{call-by-push-value} (CBPV), a calculus that  exposes the semantic structure underlying various programing paradigms~\cite{Levy:thesisbook}.  As we shall explain, it provides a form of sequencing that is more subtle than the simple kind in two ways.

For readers unfamiliar with CBPV, we explain that it has a distinction between value types $A$ (inhabited by values) and computation types $\ulb$ (inhabited by computations). For each value type $A$, there is a computation type $FA$ inhabited by  computations that aim to return a value of type $A$.  The calculus also provides other computation types, including function types.
  
A context $\Gamma$ consists of declarations $x:A$ of an identifier $x$ with value type $A$.  The judgement $\Gamma \vdashv V:A$ means that $V$ is a value of type $A$.  The judgement $\Gamma \vdashc M:\ulb$ means that $M$ is a computation of type $\ulb$. 
  
The sequencing of $M$ and $N$ is denoted $M \ttto x \ttin N$.  The first subtlety (by comparison with simple sequencing) is that $M$ may return a value, which $x$ is then bound to. The second is that $M$ has type $FA$ whereas $N$ may have any computation type $\ulb$. Meanwhile, any value $V$ gives a computation $\ttproduce V$, and this operation is the counterpart of $\mathtt{skip}$.  

The CBPV typing rules for $F$ types, returning and sequencing, along with three laws that govern these constructs, are shown in Figure~\ref{fig:Flaws}; we call this the `$F$ fragment of CBPV'. The other type and term constructors (and laws) of CBPV will not be relevant in this paper. 

\begin{figure}
    \centering
    $
    \begin{array}{c}
       \begin{prooftree}
           \Gamma \vdashv V :A 
           \justifies
           \Gamma \vdashc \ttproduce V : FA \
       \end{prooftree}
       \quad
       \begin{prooftree}
           \Gamma \vdashc M : FA \quad \Gamma, x:A \vdashc N : \ulb
           \justifies
           \Gamma \vdashc M \ttto x \ttin N : \ulb
       \end{prooftree}
       \quad
       \begin{prooftree}
           \Gamma \vdashv V : A \quad \Gamma, x:A \vdashc N : \ulb
           \justifies
           \Gamma \vdashc (\ttproduce V) \ttto x \ttin N = N [V/x] : \ulb
       \end{prooftree}
       \\ \\
       \begin{prooftree}
     \Gamma \vdashc M : FA  
     \justifies
     \Gamma \vdashc  M \ttto \ttx \ttin \ttproduce \ttx = M : FA
       \end{prooftree}
       \qquad
       \begin{prooftree}
           \Gamma \vdashc M :FA \quad \Gamma, x: A \vdashc N : FB \quad \Gamma, y : B \vdashc P : \ulc
           \justifies
           \Gamma \vdashc (M \ttto x \ttin N) \ttto y \ttin P = M \ttto x \ttin ( N \ttto y \ttin P) : \ulc
       \end{prooftree}
    \end{array}
    $
    \caption{Typing rules and laws for $F$ types in call-by-push-value.}
    \label{fig:Flaws}
    \Description[F typing rules]{Typing rules and laws for $F$ in call-by-push-value.}
\end{figure}

Although CBPV sequencing is more subtle than the simple kind, Figure~\ref{fig:Flaws} is still strikingly reminiscent of the algebraic notion of a monoid.  The three laws are clearly identifiable as left identity, right identity and associativity.  Can we make this precise?

\subsection{Mathematical Introduction} \label{sect:mathintro}

What is a monoid?  Traditionally, it is just a set equipped with an associative binary operation and an identity element.  But this is not the whole story, for in mathematics and computer science, we often encounter other notions that seem to follow a similar pattern, featuring two \emph{operations}: a binary operation identifiable as a multiplication, and a nullary operation identifiable as a unit.  Furthermore, these are required to satisfy three \emph{equational laws}, identifiable as associativity, left identity and right identity. 

Here is one example.\footnote{We use the traditional set/class distinction.  More on that later.}  Fix a class $E$---we may call its elements \emph{objects}.  A \emph{category on $E$}
 consists a family of sets of \emph{morphisms} $(\catc(x,y))_{x, y \in E}$ together with a composition function 
\begin{displaymath}
    ;_{x,y,z} \ccolc \catc(x,y) \times \catc(y,z) \rightarrow \catc(x,z)
\end{displaymath} 
for each $x,y,z \smin E$ and an identity element $\id_x \smin \catc(x,x)$ for all $x \smin E$, satisfying the usual associativity and identity laws. (From here on, we omit the subscript $x$ when it is clear to which object an identity morphism is attached.)
 
In order to encompass such notions, researchers generalized the traditional notion of monoid to that of \emph{monoid in a monoidal category} or---more general still---that of \emph{monoid in a multicategory}.  This has proved to be a successful abstract framework, encompassing such important notions as ring, monad on a given category, and category on a given class.  However, for certain purposes the framework is not general enough.  That is, there are notions that do not fit this framework but still seem to follow the \enquote{monoid} pattern.  We give several examples.
\begin{enumerate}
   \item A \emph{class span} consists of classes $A,B,E$ and functions 
   \[\sppan{A}{B}{E}{f}{g}\]  
   Say that a \emph{category} on this class span consists of a family of sets $(\catc(a,b))_{a \in A, b \in B}$, together with a composition function $;_{a,i,b} \ccolc \catc(a,gi) \times \catc(fi,b) \rightarrow \catc(a,b)$ for all $a \smin A, i \smin E, b \smin B$, and an identity $\id_i \smin \catc(fi,gi)$ for all $i \smin E$, satisfying the laws in Figure~\ref{fig:catonspanlaws}.  This seems to follow the monoid pattern.\footnote{This example is artificial---it neither appears in the literature nor has any known application.  But it illustrates the role of the bi-skew framework, while perhaps being more easily grasped than the CBPV example.}
    \item Two notions of \emph{relative monad} has appeared in the literature: a \enquote{forward} notion in the work of Altenkirch \textit{et al.}~\cite{ACU}, and a \enquote{backward} notion in the work of Spivey~\cite{Spivey}. Each of them seems to follow the monoid pattern.
    \item As explained in Section~\ref{sect:plintro}, the  CBPV treatment of returning and sequencing, displayed in Figure~\ref{fig:Flaws}, follows the monoid pattern.
\end{enumerate}

\begin{figure}
    \centering
    \begin{itemize}
        \item For any $i \smin E$ and $b \smin B$ and $l \ccolc fi \rightarrow b$, we have 
        \begin{eqnarray*}
            \id_i ; l & = & l \phantom{; \id_i}
        \end{eqnarray*}
             \item For any $a \smin A$ and $i \smin E$ and $h \ccolc a \rightarrow gi$, we have
        \begin{eqnarray*}
            h ; \id_i & = &  h \phantom{; \id_i}
        \end{eqnarray*}
        \item For any $a \smin A$ and $i,j \smin E$ and $b \smin B$, given $h \ccolc a \rightarrow gi$ and $k \ccolc fi \rightarrow gj$ and $l \ccolc fj \rightarrow b$, we have
        \begin{eqnarray*}
            (h;k);l & = & h;(k;l)
        \end{eqnarray*}
    \end{itemize}
    \caption{Laws of a category on the class span $(A,B,E,f,g)$}
    \label{fig:catonspanlaws}
    \Description[class span laws]{Laws of a category on the class span $(A,B,E,f,g)$}
\end{figure}

The goal of this paper is to generalize  the  multicategorical framework so as to encompass all of these notions.  This is done in two stages.  The first uses the notion of \emph{left-skew multicategory}~\cite{BL}.  As we shall see, the resulting framework encompasses both notions of relative monad.  There is also a dual notion of \emph{right-skew multicategory}. 

The second stage introduces the notion of \emph{bi-skew multicategory}, which generalizes both left-skew and right-skew multicategory. The progressive generalizations are depicted in Figure~\ref{fig:twocats}. As we shall see, the resulting framework encompasses the notions of category on a class span, and model of the $F$ fragment of CBPV.  Indeed, we know of no natural example of a monoid notion (whether in programming language semantics or elsewhere) that our bi-skew framework excludes.

Note that, in the associativity law of Figure~\ref{fig:catonspanlaws}, the leftmost and rightmost premises are special.  Each premise except the leftmost is a morphism from an $f$ object, and each premise except the rightmost is a morphism to  a $g$ object.  Likewise in Figure~\ref{fig:Flaws}, each premise except the leftmost has an extended context, and each premise except the rightmost has an $F$ type.  As we shall see, this reflects the essential difference between traditional multicategories and their skew variants; the latter allow various kinds of \enquote{tight} morphisms with a special leftmost and/or rightmost input.

\subsection*{Related Work}

Our work builds on that of Bourke and Lack \cite{BL}.  They present left-skew multicategory as an instance of $\catt$-multicategories, where $\catt$ is a $\Cat$-enriched operad. 
Ordinary multicategories are $\terminal$-multicategories, where $\terminal$ is the terminal $\Cat$-operad. They define an operad $\catr$ such that left-skew multicategories are $\catr$-multicategories.  Likewise, right-skew multicategories are $\catr^{\mathrm{rev}}$-multicategories, where $\catr^{\mathrm{rev}}$ is the `reverse' of $\catr$ obtained by reversing the indexing of the composition rule. 
The bi-skew multicategories of the present paper are $\catr \times \catr^{\mathrm{rev}}$-multicategories (which Bourke and Lack do not consider explicitly).

Arkor and McDermott~\cite{Arkor-McDermotupdate} have studied a different notion of skew-multicategory\footnote{The definition in version~\cite{Arkor-McDermotupdate} has been corrected in the published version~\cite{Arkor-McDermot}.} 
in which the forward notion of relative monads can be seen as monoids.  However, that approach would not work for the backward notion of \cite{Spivey}, whereas ours treats $\cato$-relative monads (for $\cato$ a bimodule) in general. See the conclusion for further details regarding their notion. 

    

\subsection*{Structure of Paper}

The paper is structured as follows. After recalling some background concepts in Section \ref{sec:prelims}, we recall the setting of (ordinary) multicategories in Section \ref{sec:multi} and some familiar examples of monoids in multicategories; many familiar examples are derived from the special case of \emph{representable} multicategories, which are those corresponding to monoidal categories. Section \ref{sec:skew} treats a step up in generality, left-skew multicategories, justified in Section \ref{ssec:relativemonad} by an examination of relative monads. Section \ref{sec:biskew} takes us to our final level of generality with bi-skew multicategories, motivated this time by an examination of CBPV in Section \ref{ssec:CBPV}.

Having accomplished the main goal of the paper, we proceed to consider the notion of \emph{unbiased} monoid, 
which includes an $n$-ary operation for each $n \in \nats$. In Section \ref{sec:unbiased}, we  formulate this in the multicategorical setting and indeed to the skew and bi-skew settings.  Happily, in each case, it is equivalent to our notion of monoid.




\begin{figure}
    \centering
\adjustbox{scale=.82,center}{%
\begin{tikzcd}[column sep = small,ampersand replacement=\&]
   \& \twotext{Monoidal categories}{and lax maps} \ar[dl] \ar[d] \ar[dr]\& \\
  \twotext{Left skew-monoidal}{categories, lax maps} \ar[d] \& \text{Multicategories} \ar[dl] \ar[dr]\& \twotext{Right skew-monoidal}{categories, lax maps} \ar[d] \\
 \twotext{Left skew}{multicategories}  \ar[dr] \&\& \twotext{Right skew}{multicategories} \ar[dl]\\
   \& \twotext{Bi-skew}{multicategories}\& 
\end{tikzcd}
}
    \caption{2-categories of multicategory variants and forgetful 2-functors between them}
    \label{fig:twocats}
    \Description[2-cat diagram]{A diagram depicting the relationship between variants of (skew-)monoidal categories and (skew-)multicategories}
\end{figure}

\section{Preliminaries}
\label{sec:prelims}

\subsection{Sets, Classes and Collections}

In category theory, \enquote{size issues} can significantly affect the structures that arise.  We shall see this e.g.\ in Proposition~\ref{prop:egrep} below.  So we explain our conventions.

This paper adopts the ontology of~\cite{AdamekHerrlichStrecker:concretecat}. 
Firstly, we distinguish the notions of \emph{set} and \emph{class}; in particular, we have the class $\mathsf{Set}$ of all sets.  We allow a still more general notion of \emph{collection}; in particular, we have the  collection of all classes, and the collection of all endofunctions on $\mathsf{Set}$.\footnote{The word ``conglomerate'' is used in~\cite{AdamekHerrlichStrecker:concretecat}.}  

This leads to two notions of category:
\begin{itemize}
    \item A \emph{light category} $\catc$ consists of a class $\obs{\catc}$ and hom-set $\catc(x,y)$ for all $x, y \in \obs{\catc}$, with composition and identities.  For example, the category $\mathbf{Set}$ of all sets.
   \item A \emph{hyperlarge category} $\catc$ consists of a collection $\obs{\catc}$ and hom-collection $\catc(x,y)$ for all $x, y \in \obs{\catc}$, with composition and identities.  For example, the category of all classes, and the category of all endofunctors on $\mathbf{Set}$.
\end{itemize}
A hyperlarge category $\catc$ is \emph{moderate} when the object collection and hom-collections are all classes, \emph{small} when they are all sets, and \emph{locally small} when the hom-collections are all sets.   Thus light is the conjunction of moderate and locally small.

The adjectives ``light'' and ``hyperlarge'' will usually be omitted.  For various categorical notions, such as  multicategory,  we define only the light version.  The hyperlarge version is then obtained by replacing ``class'' and ``set'' with ``collection''.








\subsection{Bimodules}

\begin{definition} \label{def:bimodule}
For (light) categories $\catc$ and $\catd$, a \emph{bimodule}\footnote{Also called \emph{distributor} or \emph{profunctor} (sometimes with dual variance conventions).} $\cato \ccolc \catc \relarrow \catd$ consists of the following:
\begin{itemize}
    \item For each $\catc$-object $x$ and $\catd$-object $y$, a set $\cato(x,y)$, called the ``hom-set'' from $x$ to $y$.  An element $g \smin \cato(x,y)$ is called an \emph{$\cato$-morphism} and written $g \ccolc x \rightarrow y$.
   \item For $\catc$-objects $x, y$ and a $\catd$-object $z$ and a $\catc$-morphism $f \ccolc x \rightarrow y$ and an $\cato$-morphism $g \ccolc y \rightarrow z$, the \emph{composite} $\cato$-morphism $f;g \ccolc x \rightarrow z$.
   \item For a $\catc$-object $x$ and $\catd$-objects $y,z$ and an $\cato$-morphism $g \ccolc x \rightarrow y$ and a $\catd$-morphism $h \ccolc y \rightarrow z$, the \emph{composite} $\cato$-morphism $g;h \ccolc x \rightarrow z$.
\end{itemize}
The following equations must be satisfied for any $\cato$-morphism $g \ccolc x \rightarrow y$:
\begin{itemize}
    \item \hiitem{Left and right unitality}
    \[\id_x;g = g = g;\id_y\]
    \item \hiitem{Associativity} for $f' \ccolc x'' \rightarrow x'$, $f \ccolc x' \rightarrow x$ in $\catc$ and $h \ccolc y \rightarrow y'$, $h' \ccolc y' \rightarrow y''$ in $\catd$ we have,
    \begin{align*}
        (f';f);g & = f';(f;g) \\
        (f;g);h & = f;(g;h) \\
        g;(h;h') & = (g;h);h'
    \end{align*}
\end{itemize}
\end{definition}

\noindent
A bimodule $\catc \relarrow \catd$ can be expressed in curried form as a functor $\catc^{\op} \times \catd \rightarrow \mathbf{Set}$.  However, this observation does not adapt to the setting of hyperlarge bimodules between hyperlarge categories, as there is no ``category of all collections'' to replace $\mathbf{Set}$.

\begin{definition} We provide two ways of constructing a bimodule from a functor:
\label{def:forback}
    \begin{enumerate}
        \item A functor $F \ccolc \catc \rightarrow \catd$ yields $\forbim{F} \ccolc \catc \relarrow \catd$ via $\forbim{F}(x,y) \eqdef \catd(Fx,y)$.
        \item A functor $U \ccolc \catd \rightarrow \catc$ yields $\revbim{U} \ccolc \catc \relarrow \catd$ via $\revbim{U}(x,y) \eqdef \catc(x,Uy)$.
    \end{enumerate}
\end{definition}
Note that the construction $U \mapsto \revbim{U}$ is contravariant at the 1-level, whereas the construction $F \mapsto \forbim{F}$ is contravariant at the 2-level. 

\begin{example}
An adjunction $F \dashv U \ccolc \catc \rightarrow \catd$ consists of functors $U \ccolc \catd \rightarrow \catc$ and $F \ccolc \catc \rightarrow \catd$, together with a bimodule isomorphism $\forbim{F} \cong \revbim{U}$.
\end{example}

We sometimes speak of a bimodule $\cato$ from a category to a class, or from a class to a category.  In these cases, an $\cato$-morphism can be composed at one end only.  

We can also speak of one-sided modules: for a category $\catc$, a \emph{left $\catc$-module} provides for each $\catc$-object $x$ a set $\catn(x)$ called the \enquote{hom-set} from $x$.  An element $g \smin \catn(x)$ is written $g \ccolc x \rightarrow \cdot$ and can be composed with a morphism $x' \rightarrow x$ in a way that satisfies the associativity and left-identity laws.  Such a module can be expressed as a functor $\catc^{\op} \rightarrow \Set$.


\subsection{Monoidal and Skew Monoidal Categories}

As further motivation, we recall a generalization of monoidal categories called \emph{left-skew monoidal categories}. These were originally introduced by Szlachányi in \cite{lsmonoidal} and then developed by Street and Lack in a series of papers including \cite{LackStreet1}. More recently, Uustalu, Veltri and Zeilberger proposed a sequent calculus for left-skew monoidal categories \cite{seqcalc}.

\begin{definition}
\label{def:lsmonoidal}
A \emph{left-skew monoidal category} consists of a category $\catc$, a bifunctor $\otimes \ccolc \catc^2 \rightarrow \catc$, an object $I$ and natural transformations
    \begin{spaceout}{clrcl}
        \alpha_{x,y,z} & : & (x \otimes y) \otimes z & \rightarrow & x \otimes (y \otimes z) \\
       \lambda_x & : &  I \otimes x & \rightarrow & x \\
        \rho_x & : & x & \rightarrow & x \otimes I
    \end{spaceout}%
such that the following diagrams commute:
\[\begin{tikzcd}[column sep = small,ampersand replacement=\&]
   (x \otimes I) \otimes y \ar[r, "\alpha_{x,I,y}"]\& x \otimes (I \otimes y) \ar[d,"x \otimes \lambda_y"] \\
   x\otimes y \ar[u, "\rho_x \otimes y"]  \ar[r, "\id"]\& x \otimes y
\end{tikzcd}\;
\begin{tikzcd}[ampersand replacement=\&, column sep = small]
	{(I\otimes x)\otimes y} \& {I \otimes(x \otimes y)} \& {(x \otimes y)\otimes I} \& {x \otimes (y \otimes I)} \\
	\& {x \otimes y} \& {x \otimes y}
	\arrow["{\alpha_{x,y,I}}", from=1-1, to=1-2]
	\arrow["{\lambda_x \otimes y}"', from=1-1, to=2-2]
	\arrow["{\lambda_{x \otimes y}}", from=1-2, to=2-2]
	\arrow["{\alpha_{x,I,y}}", from=1-3, to=1-4]
	\arrow["{\rho_{x \otimes y}}", from=2-3, to=1-3]
	\arrow["{x \otimes \rho_{y}}"', from=2-3, to=1-4]
\end{tikzcd}\]
\[\begin{tikzcd}[ampersand replacement=\&]
   I \otimes I \ar[r, "\lambda_I"]\& I \\
   I \ar[u,"\rho_I"] \ar[ur, "\id"]\& 
\end{tikzcd}
\begin{tikzcd}[ampersand replacement=\&]
	{((x \otimes y) \otimes z) \otimes w} \& {(x \otimes y) \otimes (z \otimes w)} \& {x \otimes (y \otimes (z \otimes w))} \\
	{(x \otimes (y\otimes z))\otimes w} \&\& {x \otimes ((y \otimes z) \otimes w)}
	\arrow["{\alpha_{x \otimes y, z, w}}", from=1-1, to=1-2]
	\arrow["{\alpha_{x,y,z} \otimes w}"', from=1-1, to=2-1]
	\arrow["{\alpha_{x, y, z\otimes w}}", from=1-2, to=1-3]
	\arrow["{\alpha_{x, y\otimes z, w}}", from=2-1, to=2-3]
	\arrow["{x \otimes \alpha_{y,z,w}}"', from=2-3, to=1-3]
\end{tikzcd}\]
\end{definition}
Given left-skew monoidal categories $\catc$ and $\catd$, a \emph{lax map} $\catc \rightarrow \catd$ is a functor $F$ together with a morphism $\varepsilon \ccolc I \rightarrow FI$ and morphisms $\mu_{x,y} \ccolc Fx \otimes Fy \rightarrow F(x \otimes y)$ natural in $x$ and $y$, compatible with the structure.  A \emph{monoidal natural transformation} $\alpha \ccolc (F,\varepsilon,\mu) \Rightarrow (F',\varepsilon',\mu')$ is a natural transformation compatible with the structure.  Thus we obtain a 2-category of left-skew monoidal categories.

Given a left-skew monoidal category, we write e.g.\ $\bigotimes_x (a,b,c)  \eqdef ((x\otimes a) \otimes b) \otimes c$.  To be precise,  $\bigotimes_{x}$ recursively sends the empty list to $x$ and a list $\lista,y$ to $(\bigotimes_{x}\lista) \otimes y$.  We abbreviate $\bigotimes_I$ as $\bigotimes$.

A \emph{monoidal category} is a left-skew monoidal category where $\alpha, \rho, \lambda$ are all isomorphisms.  In this case, the list of coherence requirements can be reduced.

\section{Basic Level of Abstraction: Multicategory}

\label{sec:multi}


\subsection{The 2-category of Multicategories}

We give the notion of multicategory, followed later by its variants.  We formulate these notions using composition at a single input; some authors prefer a formulation using simultaneous composition.
\begin{definition} \label{def:singlemulti}
    A \emph{multicategory} $\catc$ consists of the following data.
    \begin{itemize}
        \item A class $\obs{\catc}$ of objects.
        \item For any object list $\lista = (a_0,\ldots,a_{n-1})$ and object $x$, a set of morphisms $\catc(\lista;x)$. Such a $\catc$-morphism is written $f \ccolc \lista \rightarrow x$ or pictorially as:
        \[\nti*{f}{a_0, \dotsc, a_{n-1}}{x} \quad \text{ or, more concisely, } \quad \nti*{f}{~\ora{a}~}{x}.\]
        Note that the height of these morphisms is arbitrary. We use a double-headed arrow over the name of a list (for example, $\listap$) to indicate that it is assumed non-empty.
        \item \hiitem{Identities} For any object $x$, an \emph{identity} morphism:
        \[\nti*{\id}{~x~}{x}\]
        \item \hiitem{Composition} For any object lists $\listb, \listc, \listd$, objects $x,y$, morphism $f \ccolc \listc \rightarrow x$ and $g \ccolc \listb,x,\listd \rightarrow y$, a composite morphism\footnote{Strictly speaking, the composition symbol should be indexed by $\listb,\listc,\listd,x,y$, as $x$ may appear several times in the domain and homsets may not be disjoint. We rely on positioning in the diagrams to resolve this ambiguity.}:
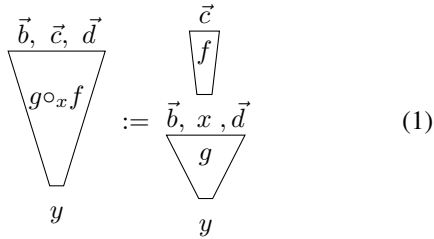
        \begin{equation}
        \label{eq:compose}
            \nti*[2]{g {\circ_x} f}{~\listb,~\listc,~\listd~}{y} ~\eqdef~ \nti*{g}{\listb,\nti{f}{~\listc~}{x},\listd}{y}
        \end{equation}
    \end{itemize}
    The following equations must be satisfied.
    \begin{itemize}
        \item \hiitem{Associativity} for $h \ccolc \lista,y,\liste \rightarrow z$ and $f,g$ as above,
    \[ \nti*[2]{h \circ_y g}{\lista,\listb, \nti{f}{~\listc~}{x}, \listd,\liste}{z}
    \nsp = \nsp \nti*{h}{\lista, \nti[2]{g{\circ_x}f}{~\listb,~\listc,~\listd~}{y}, \liste}{z}\]
    \item \hiitem{Interchange} for $f \ccolc \listb \rightarrow x$, $g\ccolc\listd \rightarrow y$ and $h\ccolc \lista,x,\listc,y,\liste \to z$:
    \[ \nti*{h \circ_x f}{\lista, \listb, \listc, \nti{g}{~\listd~}{y}, \liste}{z}
    \nsp = \nsp \nti*{h \circ_y g}{\lista, \nti{f}{~\listb~}{x}, \listc, \listd, \liste}{z}\]
    \item \hiitem{Left and right identities} for morphisms $f \ccolc \lista, x, \listb \rightarrow y$ and $g \ccolc \listc \rightarrow x$ we have the respective equations:
    \[ \nti*{f}{\lista, \nti{\id}{~x~}{x}, \listb}{y} \nsp = \nti*[2]{f}{\lista, x, \listb}{y}
    \qquad \qquad
    \nti*{\id}{\nti{g}{~\listc~}{x}}{x} \nsp = \nsp \nti*[2]{g}{~\listc~}{x}\]
    \end{itemize}
\end{definition}
We obtain a 2-category of multicategories as follows.
\begin{definition}
 Let $\catc$ and $\catd$ be multicategories.
 \begin{enumerate}
    \item A \emph{map} $F \ccolc \catc \rightarrow \catd$ sends each $\catc$-object $a$ to a $\catd$-object $Fa$, and sends each $\catc$-morphism $f \ccolc a_0,\dotsc,a_{n-1} \rightarrow b$ to a $\catd$-morphism $Ff \ccolc Fa_0,\dotsc,Fa_{n-1} \rightarrow Fb$ in a way that preserves identities and composition.
    \item Given maps $F,G \ccolc \catc \rightrightarrows \catd$, a \emph{2-cell} $\alpha \ccolc F \Rightarrow G$ sends each $\catc$-object $x$ to a $\catd$-morphism $\alpha_x \ccolc Fx \rightarrow Gx$ that is natural in the sense that, for any $\catc$-morphism $f \ccolc a_0,\ldots,a_{n-1}\rightarrow b$, the following equation holds:
    \[
        \cunarymorphism{\alpha_b}{\nti{f}{Fa_0,\ldots,Fa_{n-1}}{Fb}}[Fb]{Gb} ~=~ \nti*{Gf}{\nti{\alpha_{a_0}}{Fa_0}{Ga_0},\ldots,\nti{\alpha_{a_{n-1}}}{Fa_{n-1}}{Ga_{n-1}}}{Gb}
    \]
 \end{enumerate}
\end{definition}
The $2$-category of multicategories is denoted $\multicat$.  It has all (set-indexed) products; in particular it has a terminal object, denoted $1$.

Now we come to the main definition of the section, which has appeared e.g.\  in~\cite[page 401]{Leinster:genenrich}.
\begin{definition} \label{def:monoidmulti}
    Let $\catc$ be a multicategory.  
    \begin{enumerate}
        \item A \emph{monoid} in $\catc$ consists of a $\catc$-object $x$, a  map $m \ccolc x,x \rightarrow x$, the \emph{multiplication}, and a map $e \ccolc \: \rightarrow x$, the \emph{unit}, satisfying the following equations:
    \[\nti*{m}{\nti{e}{\nsp}{~x~},x}{x} ~=~ \nti*[2]{\id}{~x~}{x} ~=~ \nti*{m}{x, \nti{e}{\nsp}{~x~}}{x}
    \longquad
    \nti*{m}[\nti{m}{x,x}{x}]{x}{x} ~=~ \nti*{m}{x}[\nti{m}{x,x}{x}]{x}\]
    \item  For monoids $M = (x,m,e)$ and $N = (y,n,u)$, a  \emph{homomorphism} $M \rightarrow N$ is a unary morphism $f\ccolc x \rightarrow y$ satisfying the equations:
    \[ \nti*{f}{\nti{e}{\nsp}{\,x\,}}{y} ~=~ \nti*[2]{u}{\nsp}{\,y\,} \quad \quad \quad \quad \cunarymorphism{f}{\nti{m}{x,x}{x}}[\,x\,]{y} ~=~ \nti*{n}{\nti{f}{\,x\,}{y},\nti{f}{\,x\,}{y}}{y} \]%
   \item We write $\monsof{\catc}$ for the category of monoids and homomorphisms in $\catc$.
\end{enumerate}
\end{definition}
Here are some examples.
\begin{enumerate}
\item Write $\mathbf{Set}$ for the multicategory of sets, where a morphism $A_0,\dotsc, A_{n-1} \rightarrow B$ is simply a  function $A_0 \times \cdots \times A_{n-1} \rightarrow B$.  A monoid in $\mathbf{Set}$ is just a monoid in the traditional sense.
\item Write $\mathbf{Ab}$ for the multicategory of abelian groups $(X,+,0)$; a morphism $A_0,\dotsc, A_{n-1} \rightarrow B$ is a function $A_0 \times \cdots \times A_{n-1} \rightarrow B$ that preserves $+$ and $0$ in each argument. A monoid in $\mathbf{Ab}$ is just a {ring}.
\item A \emph{complete lattice} is a poset where every subset has a supremum (or equivalently, every subset has an infimum).  Write $\mathbf{SupLat}$ for the multicategory of complete lattices, where a morphism $A_0,\dotsc, A_{n-1} \rightarrow B$ is a function $A_0 \times \cdots \times A_{n-1} \rightarrow B$ preserving suprema in each argument.  A monoid in $\mathbf{SupLat}$ is a \emph{quantale}.
\item A \emph{set-valued matrix} $X \ccolc A \relarrow B$, for classes $A$ and $B$, is a family of sets $(X(a,b))_{a \in A, b\in B}$.  For a class $E$, write  $\setval_E$ for the multicategory of set-valued endomatrices on $E$, where a morphism $X_0,\ldots,X_{n-1} \to Y$ is a function sending each list $a_0,\ldots,a_{n} \smin E$ and list $x_0,\dotsc,x_{n-1}$ with $x_i \in X_i(a_i,a_{i+1})$ to an element of $Y(a_0, a_n)$.  A monoid in $\setval_E$ is just a category on $E$ as described in Section \ref{sect:mathintro}.\footnote{Put differently, a category is a \emph{monad} in the hyperlarge \emph{virtual bicategory} of classes and set-valued matrices.  Likewise, a {small} category is a monad in the bicategory of sets and small-valued matrices, or equivalently in $\Span(\Set)$.}
\end{enumerate}

Any monoidal category $\catc$ gives a multicategory $\monmult{\catc}$ with the same objects: a morphism $f \ccolc \lista \to b$ is a $\catc$-morphism $\bigotimes \lista \rightarrow b$.  Hermida~\cite[\S9.2]{Hermida} showed that the multicategories so obtained are the \emph{representable} ones.
\begin{definition}
\label{def:representable}
In a multicategory, a \textbf{tensor} of a list of objects $\listb$ consists of an object $v$ and morphism $p \ccolc \listb \to v$ such that, for any lists of objects $\lista$ and $\listc$, object $x$ and morphism $f \ccolc \lista, \listb, \listc \to x$, there exists a unique morphism $g \ccolc \lista,v,\listc \to x$ with:
\[
    \nti*[2]{f}{\lista,~\listb,~\listc}{x}
    ~=~
    \nti*{g}{~\lista,\nti{p}{~\listb~}{v},~\listc~}{x}
\]
Note that the tensor $(v,p)$ is unique up to unique isomorphism.  A multicategory is \textbf{representable} when every list of objects has a specified tensor.
\end{definition}
 
\begin{proposition} \label{prop:egrep}
Considering the above examples, the multicategories $\Set$, $\mathbf{Ab}$ and $\mathbf{SupLat}$ are representable, with tensors given by cartesian products in $\Set$ and respective tensor products in $\mathbf{Ab}$ and $\mathbf{SupLat}$. For a class $E$, the multicategory $\setval_E$ is representable if and only if $E$ is a set.
\end{proposition}
\begin{proof}
Representability of $\Set$, $\mathbf{Ab}$ and $\mathbf{SupLat}$ is well-known, as is that of $\setval_E$ for a set $E$.

Conversely, suppose $\setval_E$ representable.
If $E$ is empty, then it is a set; otherwise, $E$ has an element $e$.  Define matrices $A \eqdef (1_{x=e})_{x,y \in E}$ and $B \eqdef (1_{z=e})_{y,z \in E}$.  By hypothesis, the list $A,B$ has a tensor $(C,f)$.  We show that the function $E \rightarrow C(e,e)$ sending $x \mapsto f_{e,x,e}(*,*)$ is injective.  Since $C(e,e)$ is a set, it follows that $E$ is too.

Let $D$ be the matrix $(2)_{x, y \in E}$.  For any $r,s \smin E$ such that $f_{e,r,e}(*,*) =f_{e,s,e}(*,*)$, let $g \ccolc A,B \rightarrow D$ be the map that at $e,x,e$ sends $(*,*)$ to $0$ if $x=r$ and $1$ otherwise.  We know that $g$ factorizes as $f$ followed by some $h \ccolc C \rightarrow D$, so we have
\[g_{e,s,e}(*,*) = h_{e,e}(f_{e,s,e}(*,*)) = h_{e,e}(f_{e,r,e}(*,*)) = g_{e,r,e}(*,*) = 0 \]
so $s=r$ as required.
\end{proof}
A conclusion of Proposition \ref{prop:egrep} is that while many familiar examples of multicategories are representable, size issues may force us to consider multicategories which are not representable.

Lastly in this section, we note that \emph{any} category $\catc$ is a category of monoids.  To this end, we write $\seqof{\catc}$ for the multicategory\footnote{This is called a \emph{sequential multicategory} in~\cite{Pisani}, since a morphism is a sequence.} with the same objects as $\catc$, where a morphism from $\lista = (a_0,\ldots,a_{n-1})$ to object $x$ is a list of $\catc$-morphisms $(f_i \ccolc a_i \rightarrow x)_{i < n}$.  Composition and identities are inherited from $\catc$ in the obvious way.  
\begin{proposition}
    For any category $\catc$, the forgetful functor $\monsof{\seqof{\catc}} \rightarrow \catc$ is an isomorphism.
\end{proposition}
\begin{proof}
    Each object $x$ carries a unique monoid structure in $\seqof{\catc}$. Indeed, there is only one possible choice for with unit, the empty list $()$. Given a prospective binary multiplication morphism $(a,b)$, the unit axioms then impose that $a = \id_x = b$, so $(\id_x,\id_x)$ is the unique possibility.  We also see that each $\catc$-morphism $x \rightarrow y$ is a homomorphism. 
\end{proof}

A tensor for a list in $\seqof{\catc}$ is easily seen to be a coproduct of the objects in the list, so $\seqof{\catc}$ is representable if and only if $\catc$ has finite coproducts, so we recover the following well-known result.

\begin{corollary}
Let $\catc$ have finite coproducts. The forgetful functor $\monsof{\catc} \rightarrow \catc$ is an isomorphism.
\end{corollary}



\section{Intermediate Level: Skew Multicategory}
\label{sec:skew}

\subsection{Digression: Relative Monads}
\label{ssec:relativemonad}

We leave the multicategory story for a while, to learn about the notion of relative monad.  (Later, we shall explain how to see this as another monoid example.) We present it in a way that unifies two versions of this notion from the literature.

Recall first that a monad on a category $\catc$ may (especially in computer science) be given as a \emph{Kleisli triple} $(T,\eta,{}^{*})$, where $T$ is given on objects only, $\eta$ is a family of maps $\unitof{x} \ccolc x \to Tx$ and ${}^{*}$ sends any morphism $f \ccolc x \rightarrow Ty$ to its \emph{Kleisli extension} $f^{*} \ccolc Tx \rightarrow Ty$.  The following definition is arranged in a similar way. 
\begin{definition} \label{def:relmonadclass}
   Let $\cato$ be a bimodule from a category $\catc$ to a category $\catd$.  An \emph{$\cato$-relative monad} $(T,\eta,{}^*)$ consists of the following data:
      \begin{itemize}
    \item For each $x \smin \catc$, a $\catd$-object $Tx$ and $\cato$-morphism $\unitof{x} \ccolc x \rightarrow Tx$.
    \item For each $x,y \smin \catc$ and $\cato$-morphism $k \ccolc x \rightarrow Ty$, a $\catd$-morphism $\exof{k} \ccolc Tx \rightarrow Ty$.
    \end{itemize}
    The following conditions must be satisfied.
    \begin{itemize}
        \item For each $x,y \smin \catc$ and $k \ccolc x \rightarrow Ty$, this triangle commutes:
        \[ \begin{tikzcd}[ampersand replacement=\&]
        x \arrow[r,"\eta_x"] \ar[dr,"k"']\& Tx \ar[d,"k^{*}"] \\
        \& Ty
        \end{tikzcd} \]%
    \item For each $x \smin \catc$, the parallel pair
    \begin{tikzcd}[ampersand replacement=\&]
        Tx \ar[r, bend left, "\eta_{x}^{*}"] \ar[r, bend right, "\id"']\& Tx
    \end{tikzcd} commutes, which is to say $\eta_{x}^* = \id_x$.
    \item For $x,y,z \smin \catc$ and $k \ccolc x \rightarrow Ty$ and $l \ccolc y \rightarrow Tz$, this triangle commutes:
    \[ \begin{tikzcd}[ampersand replacement=\&]
        Tx \arrow[r,"k^{*}"] \ar[dr,"(k;l^{*})^{*}"']\& Ty \ar[d,"l^{*}"] \\
       \& Tz
    \end{tikzcd} \]%
    \end{itemize}
\end{definition}

Note that $\catc$-morphisms are ignored in this definition---we could even take $\cato$ to be a bimodule from a \emph{class} to a category.  Although ignoring $\catc$-morphisms may seem odd, the following result justifies it.
\begin{proposition} \label{prop:extrelmonad}
   For any $\cato$-relative monad  $(T,\eta,^{*})$, the function $T$ uniquely extends to a functor $\catc \rightarrow \catd$ making $\eta$ and $^{*}$ natural in the sense that for $\catc$-morphisms $f:u \to x$, $g:x \to Ty$ and $h:y \to z$, the following diagrams commute:
    \[ \begin{tikzcd}[ampersand replacement=\&]
            u \ar[r, "\eta_u"] \ar[d, "f"']\& Tu \ar[d, "Tf"] \\
            x \ar[r, "\eta_x"']\& Tx
    \end{tikzcd}
    \qquad
    \begin{tikzcd}[ampersand replacement=\&]
      Tu \ar[r,"Tf"] \ar[dr,"(f;g)^{*}"']\& Tx \ar[d,"g^{*}"]\\
     \& Ty
    \end{tikzcd}
    \qquad
    \begin{tikzcd}[ampersand replacement=\&]
      Tx \ar[r,"g^{*}"] \ar[dr,"(g;Th)^*"']\& Ty \ar[d,"Th"] \\
     \& Tz.
    \end{tikzcd} \]
\end{proposition}
\begin{proof}
A functor extending the function $T$ must send $k \ccolc x \rightarrow y$ to $(k;\eta_y)^{*}$ by commutativity of
\[\begin{tikzcd}[ampersand replacement=\&]
    Tx \ar[dr,"Tk"'] \ar[rr,"(k;\eta_y)^{*}"]\&\& Ty \\
   \& Ty \ar[ur, "\eta_x^{*}"] \ar[ur, bend right, "\id"'] 
\end{tikzcd}.\]
Conversely, for $k\ccolc x \rightarrow y$, define $Tk \ccolc Tx \rightarrow Ty$ to be $(k;\eta_y)^{*}$.  We verify the requirements:
\begin{itemize}
    \item $T$ preserves the identity at $x \smin C$ by commutativity of 
\[\begin{tikzcd}[ampersand replacement=\&]
    Tx \ar[r, bend left=40, "T\id"] \ar[r, "\eta_x^{*}"] \ar[r, bend right=40, "\id"']\& Tx.
\end{tikzcd}\]%
\item $T$ preserves the composite $\begin{tikzcd}[ampersand replacement=\&]x \ar[r, "k"]\& y \ar[r, "h"]\& z
\end{tikzcd}$%
by commutativity of 
\[\begin{tikzcd}[ampersand replacement=\&]
	Tx \&\&\&\& Ty \\
	\\
	\&\&\&\& Tz
	\arrow["{(k;\eta_y)^*}", from=1-1, to=1-5]
	\arrow["Tk", bend left = 20, from=1-1, to=1-5]
	\arrow["{(k;\eta_y;(h;\eta_y)^*)^*}"{pos=0.4}, from=1-1, to=3-5]
	\arrow["{(k;h;\eta_z)^*}"{description}, bend right = 15, from=1-1, to=3-5]
	\arrow["{T(k;h)}"', bend right = 30, from=1-1, to=3-5]
	\arrow["{(h;\eta_z)^*}"', from=1-5, to=3-5]
	\arrow["Th", bend left = 20, from=1-5, to=3-5]
\end{tikzcd}\]
\item The square commutes by commutativity of 
\[\begin{tikzcd}[ampersand replacement=\&]
	u \&\& Tu \\
	x \&\& Tx
	\arrow["{\eta_u}", from=1-1, to=1-3]
	\arrow["f"', from=1-1, to=2-1]
	\arrow["{(f;\eta_x)^*}"', from=1-3, to=2-3]
	\arrow["Tf", bend left = 30, from=1-3, to=2-3]
	\arrow["{\eta_x}"', from=2-1, to=2-3]
\end{tikzcd}\]
\item The triangles commute by commutativity of 
\[\begin{tikzcd}[column sep = 4em,ampersand replacement=\&]
    Tu \ar[rr, bend left=20, "Tf"] \ar[rr, "(f;\eta_x)^{*}"]\ar[ddrr, bend right=15, "(f;g)^{*}"'] \ar[ddrr, "(f;\eta_x;g^{*})^{*}"]
\&\& Tx \ar[dd, "g^{*}"] \\
   \&\& \\
   \&\& Ty,
\end{tikzcd}
\quad \text{and} \quad
\begin{tikzcd}[column sep = 4em,ampersand replacement=\&]
    Tx \ar[rr, "g^{*}"] \ar[ddrr, bend right=15, "(g;Th)^{*}"'] \ar[ddrr, "(g;(h;\eta_z)^{*})^{*}", pos=0.4]\&\& Ty \ar[dd ,"(h;\eta_z)^{*}"'] \ar[dd, bend left, "Th"] \\
   \&\& \\
   \&\& Tz.\& \qedhere 
\end{tikzcd} \]
\end{itemize}
\end{proof}

Now we recover the two versions from the literature as special cases via the constructions of bimodules from functors in Definition \ref{def:forback}.

Let $\catc$ and $\catd$ be categories and consider a functor $J \ccolc \catc \rightarrow \catd$. $\forbim{J}$-relative monads are exactly the \enquote{relative monads} introduced by Altenkirch, Chapman and Uustalu in~\cite{ACU}.  For example, let $J$ be the embedding of the small category of (a skeleton of) finite sets and functions into $\Set$.  Then a $\forbim{J}$-relative monad, also called a \enquote{cartesian operad} or \enquote{abstract clone}, corresponds to a finitary monad $T$ on $\Set$.  We think of $T([n])$ as the set of all $n$-ary operations.  This example generalizes to the notion of \enquote{monad with arities}~\cite{arities}.

For a functor $U \ccolc \catd \rightarrow \catc$, $\revbim{U}$-relative monads are exactly the kind introduced by Spivey in~\cite{Spivey}.  To illustrate this case, let $U$ be the forgetful functor from $\mathbf{Grp}$ (the category of all groups) to $\Set$.  Since $U$ is faithful, we see that a $\revbim{U}$-relative monad is a Kleisli triple (or monad) on $\Set$ equipped with a group structure on $TA$ for each set $A$, such that the Kleisli extension of each function $A \rightarrow TB$ is a homomorphism $TA \rightarrow TB$. 
Since this functor $U$ has a left adjoint $F$ (viz.\ the free group functor), the notion of $\revbim{U}$-relative monad coincides with that of $\forbim{F}$-relative monad. To show that the backward-relative monad really is an independent notion, let us see an example where $U$ has no left adjoint.

Define a \emph{dynamical complete lattice} to be a complete lattice 
equipped with an endofunction, 
called its \emph{dynamics}.  Write $\dynsup$ for the category of all dynamic complete lattices and \emph{sup-homomorphisms}, i.e.\ functions that preserve all suprema and commute with the dynamics.

\begin{proposition} \label{prop:noinit}
    $\dynsup$ has no initial object.
\end{proposition}
\begin{proof}
Given an initial object $(C,\zeta)$, Hartogs' lemma yields an ordinal $\alpha$ such that there is no injection $\alpha \rightarrow C$. Form the dynamic complete lattice $R$ given by the set of ordinals $\leqslant \alpha$, with dynamics sending $\beta < \alpha$ to $\mathsf{S}\beta$, and $\alpha$ to $\alpha$.  There is a unique sup-homomorphism $h \ccolc (C,\zeta) \rightarrow R$.  Each $\beta \leqslant \alpha$ has an  $h$-preimage $x_{\beta}$, defined recursively via
\begin{spaceout}{rcll}
      x_{\beta} & \eqdef &  \bigvee_{\gamma < \beta} x_{\gamma} & \text{($\beta$ is 0 or a limit)} \\
    x_{\mathsf{S}\beta} & \eqdef & \zeta(x_{\beta}) & \text{otherwise}
\end{spaceout}
Since $\beta \mapsto x_{\beta}$ is a section of $h$, it is injective, contradiction.
\end{proof}
By Proposition~\ref{prop:noinit}, the forgetful functor $U \ccolc \dynsup \rightarrow \Set$ has no left adjoint, as left adjoints preserve initiality. Nonetheless, the story is the same as before: a $\revbim{U}$-relative monad is a Kleisli triple (or monad) on $\Set$  equipped with a dynamical suplattice structure on $TA$ for each set $A$, such that the Kleisli extension of each function $A \rightarrow TB$ is a sup-homomorphism $TA \rightarrow TB$.  For example, let $T$ send each set $A$ to the free (with respect to sup-homomorphisms) dynamic complete lattice on $A$ with dynamics preserving suprema.  Explicitly, this is $\pset(\nats \times A)$ ordered by inclusion,  with dynamics sending $c \smin \pset(\nats \times A)$ to $\setbr{(n+1,x) \mid (n,x) \smin c}$ and unit sending $x \smin A$ to $\setbr{(0,x)}$.

 The following \enquote{naturality} formulation appeared in~\cite{Spivey}.
\begin{proposition} \label{prop:natbackward}
    For a functor $U \ccolc \catd \rightarrow \catc$, a $\revbim{U}$-relative monad can be represented as a functor $T \ccolc \catc \rightarrow \catd$ together with natural transformations  $\eta \ccolc \id_{\catc} \rightarrow UT$ and $\mu \ccolc TUT \rightarrow T$ such that the following diagrams commute:
\[\begin{tikzcd}[column sep = small,ampersand replacement=\&]
    T \ar[r, "T\eta"] \ar[dr, "\id_T"']\& TUT \ar[d, "\mu"]\& UT \ar[r, "\eta UT"] \ar[dr, "\id_{UT}"']\& UTUT \ar[d, "U\mu"] 
   \& TUTUT \ar[r, "\mu UT"] \ar[d, "TU\mu"']\& TUT \ar[d, "\mu"] \\
    \& T\&\& UT\& TUT \ar[r, "\mu"]\& T
\end{tikzcd}\]
\end{proposition}



\subsection{Monoids in a Skew Multicategory}

We now present Bourke and Lack's notion of left-skew multicategory~\cite{BL}, which uses two kinds of morphism, dubbed \enquote{tight} and \enquote{loose}.  Perhaps surprisingly, this will allow us to view $\cato$-relative monad as a monoid notion---Theorem~\ref{thm:bimodule} below.

Recall that we write $\lista, \listb, \ldots$ for possibly empty object lists, and $\listap, \listbp, \ldots$ for nonempty object lists.
\begin{definition}
    A \emph{left-skew multicategory} $\catc$ consists of the following data.
    \begin{itemize}
        \item A class  $\obs{\catc}$ of objects.
        \item For any object list $\lista$ and object $x$, a set of \emph{loose} morphisms $\catc(\lista;x)$.  Such a morphism is written $f \ccolc \lista \rightarrow x$ or pictorially as 
            \[\nti{f}{~\ora{a}~}{x}\]
        \item For any nonempty object list $\listap$ and object $x$, a set of \emph{tight} morphisms $\catc(\lti{\listap};x)$.  Such a morphism is written $f \ccolc \lti{\listap} \rightarrow x$ or pictorially as
            \[\ntile*{f}{\listap}{x}\]
        \item \hiitem{Loosening} An operation sending each tight morphism $f \ccolc \lti{\listap} \rightarrow x$ to a loose morphism $\lloo{f} \ccolc \listap \rightarrow x$:
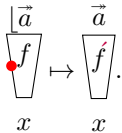
            \[ \ntile*{f}{\listap}{x} \mapsto \nti*{\lloo{f}}{~\listap~}{x}.\]
        \item \hiitem{Tight identities} For any object $x$, a morphism
            \[\ntile*{\id}{x}{x}.\]
        \item \hiitem{Loose composition (A)} For any objects lists $\lista, \listb, \listc$ and objects $x,y$ and (loose) morphisms $f \ccolc \listb \rightarrow x$ and $g \ccolc \lista,x,\listc \rightarrow y$, a composite 
        \[\nti*[2]{g {\circ_x} f}{~\lista,~\listb,~\listc~}{y} ~\eqdef~
        \nti*{g}{\lista,\nti{f}{~\listb~}{x},\listc}{y}\]
        \item \hiitem{Tight composition (B)} For any nonempty object list $\listap$, object list $\listb$, objects $x,y$ and morphisms $f \ccolc \lti{\listap} \rightarrow x$ and $g \ccolc \lti{x},\listb \rightarrow y$, a composite
            \[  \ntile*[2]{g {\circ_x} f}{\listap,\nsp\listb~}{y} ~\eqdef~
            \ntile*{g}[\ntile{f}{\listap}{x}]{,\listb~}{y} \]
        \item \hiitem{Mixed composition (C)} For any nonempty object list $\listap$, objects lists $\listb,\listc$, objects $x,y$ and morphisms $f \ccolc \listb \rightarrow x$ and $g \ccolc \lti{\listap},x,\listc \rightarrow y$, a composite
            \[ \ntile*[2]{g {\circ_x} f}{\listap,~\listb,~\listc}{y} ~\eqdef~
            \ntile*{g}{\listap,\nti{f}{~\listb~}{x},\listc}{y} \]
        \end{itemize}
        The following conditions must be satisfied.
        \begin{itemize}
        \item Loosening commutes with tight composition and with mixed composition.  
        \item Left identity for each kind of composition: \emph{A}, \emph{B}, \emph{C}.
        \item Right identity for each kind of morphism: loose and tight.
         \item Associativity laws:
         \begin{math}
             \as{A}{A}, \as{B}{B}, \as{A}{C}, \as{A}{B}.
         \end{math} 
         \item Interchange laws: 
         \begin{math}
             \ch{A}{A}, \ch{B}{C}, \ch{C}{C}.
         \end{math}
        \end{itemize}
        Letters \emph{A}, \emph{B}, \emph{C} indicate the composition forms. Here and in Definition \ref{def:biskew}, the associativity and interchange laws impose that either order in which the respective forms of composition may be performed produces the same result.
        \end{definition}

Any left skew-monoidal category $(\catc,\otimes,I,\alpha,\lambda,\rho)$ gives a left skew multicategory with the same objects: a loose morphism $f \ccolc \lista \to b$ is a $\catc$-morphism $\bigotimes \lista \rightarrow b$, and a tight morphism $f \ccolc \lti{x,\lista} \to b$ is a $\catc$-morphism $\bigotimes_x \lista \rightarrow b$ (using notation from Definition \ref{def:lsmonoidal}).  Loosening is defined by composing with the left unitor $\lambda_x: I \otimes x \to x$ in the first argument.  Bourke and Lack showed that the left skew multicategories so obtained are the \emph{left-representable} ones~\cite[Theorem 6.1]{BL}.

\begin{definition}
\label{def:lerep}
In a left-skew multicategory, a \emph{tight left tensor} of a non-empty list of objects $\listap$ consists of an object $v$ and tight morphism $p \ccolc \lti{\listap} \to v$ such that, for any lists of objects $\listb$, object $x$ and morphism $f \ccolc \lti{\listap}, \listb \to x$, there exists a unique morphism $g \ccolc \lti{v},\listb \to x$ with:
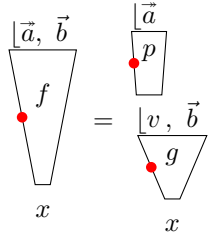
\[
    \ntile*[2]{f}{\listap,~\listb}{x}
    ~=~
    \ntile*{g}[\ntile{p}{\listap}{v}]{,~\listb~}{x}
\]
Similarly, a \emph{loose left tensor} of a list $\lista$ consists of an object $v$ and a loose morphism $p \ccolc \lista \to v$ such that, for any lists of objects $\listb$, object $x$ and morphism $f \ccolc \lista, \listb \to x$, there exists a unique tight morphism $g \ccolc \lti{v},\listb \to x$ with:
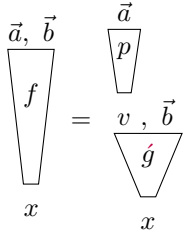
\[
    \nti*[2]{f}{\lista,~\listb}{x}
    ~=~
    \nti*{\lloo{g}}[\nti{p}{~\lista~}{v}]{,~\listb~}{x}
\]
Note that this is weaker than the condition imposed in Definition \ref{def:representable}, but that these left tensors are nonetheless unique up to unique tight isomorphism.  A left-skew multicategory is \textbf{left-representable} when every nonempty object list has a tight left tensor and every object list a loose left tensor.
\end{definition}

\begin{remark}
Offord provides string diagram syntax for left-skew multicategories in \cite{offord} which is essentially equivalent to our trapezate syntax. Offord also includes syntactic operations for tensors in order to specialize to the representable case.
\end{remark}
        
Now let us assemble a 2-category of left-skew multicategories.  Given left-skew multicategories $\catc$ and $\catd$, a \emph{map} $F \ccolc \catc \rightarrow \catd$ sends objects to objects, loose morphisms to loose morphisms, and tight morphisms to tight morphisms in the obvious way.  Given maps  $F,G \ccolc \catc \rightrightarrows \catd$, a \emph{2-cell} $\alpha \ccolc F \Rightarrow G$ sends each $\catc$-object $x$ to a tight map $\alpha_x \ccolc \lti{Fx} \to Gx$, satisfying the evident naturality condition for every (loose or tight) $\catc$-morphism.

We may now generalize the notion of monoid (Definition~\ref{def:monoidmulti}) to the skew setting.
\begin{definition}
\label{def:leftskewmonoids}
    Let $\catc$ be a left-skew multicategory.  
    \begin{enumerate}
        \item A \emph{monoid} in $\catc$ consists of a $\catc$-object $x$, a  tight map $m \ccolc \lti{x,x} \to x$ (the \emph{multiplication}), and a nullary map $e \ccolc \to x$ (the \emph{unit}), satisfying the following equations.
        
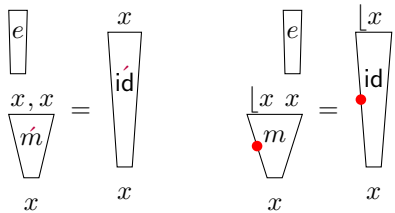
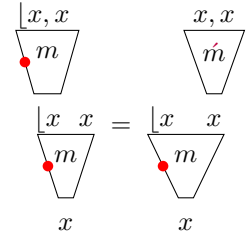
\[ \nti*{\lloo{m}}{\nti{e}{\nsp}{~x~},x}{x} ~=~ \nti*[2]{\lloo{\id}}{~x~}{x}
\longquad
\ntile*{m}{x,}[\nti{e}{\nsp}{~x~}]{x} ~=~ \ntile*[2]{\id}{x}{x}
\longquad
\ntile*{m}[\ntile{m}{x,x}{x}]{, x}{x} ~=~ \ntile*{m}{x,}[\nti{\lloo{m}}{~x,x~}{x}]{x} \]%
    \item  For monoids $M = (x,m,e)$ and $N = (y,n,u)$, a  \emph{homomorphism} $M \rightarrow N$ is a tight map $f \ccolc \lti{x} \to y$ satisfying
    \[ \nti*{\lloo{f}}{\nti{e}{\nsp}{~x~}}{y} ~=~ \nti*[2]{u}{\nsp}{~y~} \quad \quad \quad \quad
    \ntile*{f}[\ntile{m}{x,x}{x}]{\hspace{-3pt}}{y} ~=~ \ntile*{n}[\ntile{f}{x}{y}]{,~\nti{\lloo{f}}{~x~}{y}}{y} \]%
  \item We write $\monsof{\catc}$ for the category of monoids and homomorphisms in $\catc$.
    \end{enumerate}
\end{definition}

\subsection{Examples of Monoids in Left-skew Multicategories}

First we explain how multicategories are a special case of left-skew multicategories.
\begin{definition} \label{def:leftskewify}
Any multicategory $\catc$ gives a left-skew multicategory $\leftskewify (\catc)$ with the same objects.  Both the loose and the tight hom-sets are just $\catc$-hom-sets, and loosening is identity.  Composition and identities are inherited from $\catc$.
We extend $\leftskewify$ to a 2-functor in the obvious way.
\end{definition}

\begin{proposition} \label{prop:leftskewify}
For a multicategory $\catc$, $\monsof{\catc}$ is isomorphic to $\monsof{\leftskewify (\catc)}$.
\end{proposition}


Now we turn to our main example of a monoid in a left-skew multicategory:
    \begin{theorem}
    \label{thm:bimodule}
    Let $\cato$ be a bimodule from a category $\catc$ to a category $\catd$.  Then an $\cato$-relative monad is precisely a monoid in the following hyperlarge left-skew multicategory.
    \begin{itemize}
    \item An object is a functor $\catc \rightarrow {\catd}$.
    \item A morphism $f \ccolc T_0,\ldots,T_{n-1} \rightarrow H$ is an extranatural\footnote{See~\cite{extranat} for the  definition of extranaturality.} family of maps
    \begin{displaymath}
        \cato(x_{0},T_{0} x_{1}) \times \cdots \times \cato(x_{n-1},T_{n-1} x_{n}) \rightarrow 
    \cato(x_{0}, Hx_{n})
    \end{displaymath}
    \item A morphism $f \ccolc \lti{S, T_0,\ldots,T_{n-1}} \rightarrow H$ is an extranatural family of maps
    \begin{displaymath}
        \cato(x_{0},T_{0} x_{1}) \times \cdots \times \cato(x_{n-1},T_{n-1} x_{n}) \rightarrow 
    \catd(Sx_{0}, Hx_{n})
    \end{displaymath}
    \item Loosening, identities and composition are defined in the evident way.
    \end{itemize}
\end{theorem}
\begin{proof}
    Immediate from Proposition~\ref{prop:extrelmonad}.
\end{proof}

Consider the above left-skew multicategory in the case that $\cato$ has the form $\revbim{U}$ for a functor $U \ccolc \catd \rightarrow \catc$.  The Yoneda lemma allows us to represent:
\begin{itemize}
    \item a morphism $f \ccolc T_0,\ldots,T_{n-1} \rightarrow H$ as a natural transformation $UT_0\cdots UT_{n-1} \rightarrow UH$
    \item a morphism $f \ccolc \lti{S, T_0,\ldots,T_{n-1}} \rightarrow H$ as a natural transformation $SUT_0 \cdots UT_{n-1} \rightarrow H$
\end{itemize}
with loosening, identities and composition given in the evident way.  
This immediately yields Proposition~\ref{prop:natbackward}.

\section{Advanced Level: Bi-skew Multicategory}
\label{sec:biskew}

\subsection{Digression: Call-By-Push-Value Sequencing}
\label{ssec:CBPV}

We leave the multicategory generalization story for a while, and turn to the $F$ fragment of CBPV (introduced in Section~\ref{sect:plintro}) and its categorical semantics.  Later we shall explain how we see this categorical structure as another monoid example.

To formulate the $F$ fragment as a formal calculus, we require a set of base value types and a set of base computation types.  We then say that a value type is just a base value type, and a computation type (underlined) is either a base computation type or $FA$ for a value type $A$.  We next require
\begin{itemize}
\item a set of value constructors $f \ccolc (A_i)_{i < n} \rightarrow B$
\item a set of computation constructors $g \ccolc  (A_i)_{i < n}\rightarrow \ulb$.
\end{itemize}
The colllection of base types and term constructors shall be called a \emph{CBPV signature}, written $\cats$.  It gives rise to a calculus, via the typing rules in Figure~\ref{fig:Flaws} and additionally
\begin{displaymath}
    \begin{array}{c}
         \begin{prooftree}
             (\Gamma \vdashv V_i : A_i)_{i < n}
             \using f \ccolc (A_i)_{i < n} \rightarrow B
             \justifies
             \Gamma \vdashv f(V_i)_{i < n} : B
         \end{prooftree} \longquad
         \begin{prooftree}
             (\Gamma \vdashv V_i : A_i)_{i < n}
             \using g \ccolc (A_i)_{i < n} \rightarrow \ulb
             \justifies
             \Gamma \vdashc g(V_i)_{i < n} : \ulb
         \end{prooftree}
    \end{array}
\end{displaymath}
We call this calculus $\retseq{\cats}$.  For contexts $\Gamma$ and $\Delta$, a \emph{substitution} $k \ccolc \Gamma \rightarrow \Delta$ assigns to each identifier declaration $(x:A) \smin \Gamma$ a value $\Delta \vdashv k_x : A$.  For any such substitution $k$ and term $M$ in context $\Gamma$, we obtain a term $k^*M \eqdef M[k_x/x]_{(x:A) \in \Gamma}$ in context $\Delta$.

Next we give categorical structure that models $\retseq{\cats}$.  This is very different from the CBPV categorical semantics in the literature~\cite{Levy:thesisbook}, both because we are modelling only a small fragment of CBPV, and because we seek to crudely mimic  the syntactic description.  

Our categorical semantics is given in two parts.  The first part models types, contexts,  values and substitutions, and the second computations.

First, a \emph{cartesian base} for CBPV modelling consists of a cartesian category $\catc$, classes $X$ and $Y$, and functions 
\[\sppan{\obs{\catc}}{Y}{X}{j}{F}\]
Our intended interpretation is as follows.
\begin{itemize}
\item A value type $A$ will denote an element $\seman{A} \smin X$.
\item A computation type $\ulb$ will denote an element $\ulb \smin Y$, using
\begin{eqnarray*}
    \seman{FA} & = & F\seman{A} \phantom{F}
\end{eqnarray*}
\item A typing context $\Gamma$ will denote a $\catc$-object $\seman{\Gamma}$, using
\begin{eqnarray*}
    \seman{\mathsf{nil}} & = & 1 \\
    \seman{\Gamma, x:A} & = & \seman{\Gamma} \times j \seman{A} 
\end{eqnarray*}
\item A value $\Gamma \vdashv V:A$ will denote a $\catc$-morphism $\seman{\Gamma} \rightarrow j\seman{A}$.  
\item A substitution $k \ccolc \Gamma \rightarrow \Delta$ will denote a $\catc$-morphism $\seman{k} \ccolc \seman{\Delta} \rightarrow \seman{\Gamma}$.
\end{itemize}
The second part of our semantics closely resembles Figure~\ref{fig:Flaws}.
\begin{definition} \label{def:seqmodel}
    On a cartesian base $(\catc,X,Y,j,F)$, a \emph{RetSeq model} consists of the following data.
    \begin{itemize}
        \item A bimodule $\cato$ from $\catc$ to $Y$.
    \item For each $A \smin X$, a function
    \begin{displaymath}
 \eretof{\Gamma}{A} \ccolc      \catc(\Gamma, jA)  \rightarrow \cato(\Gamma, FA)
 \end{displaymath}
    natural in $\Gamma \smin \catc$.
    \item For each $A \smin X$ and $\ulb \smin Y$, a  function 
    \begin{displaymath}
    \etoof{\Gamma}{A}{\ulb} \ccolc    \cato(\Gamma, FA) \times \cato(\Gamma \times jA, \ulb) \rightarrow \cato(\Gamma, \ulb)
    \end{displaymath}
natural in $\Gamma\smin \catc$. 
    \end{itemize}
      The following laws must be satisfied.
   \begin{itemize}
       \item For any $\catc$-morphism $V \ccolc \Gamma \rightarrow j(A)$ and $\cato$-morphism $M \ccolc \Gamma \times jA \rightarrow \ulb$, 
       we have
       \begin{eqnarray*}
       (\retof{\Gamma}{A} \af V)  \af\toof{\Gamma}{A}{\ulb} \af M  & = & 
\tuple{\id_{\Gamma},V} \,;\,M \phantom{\pi}
       \end{eqnarray*}
  \item For any $\cato$-morphism $M \ccolc \Gamma \rightarrow FA$, we have 
  \begin{eqnarray*}
      M \af \toof{\Gamma}{A}{FA}\af(\retof{\Gamma\times A}{A} \af \pipof{\Gamma}{A})  & = & M \phantom{\af \toof{\Gamma}{A}{FA}\af(\retof{\Gamma\times A}{A} \af \pipof{\Gamma}{A})}
  \end{eqnarray*}
  \item For any $\cato$-morphisms $M \ccolc \Gamma \rightarrow FA$ and $N \ccolc \Gamma \times jA\rightarrow FB$ and $P \ccolc \Gamma \times jB \rightarrow \ulc$, we have
\begin{eqnarray*}
    \phantom{F(\pi_{\Gamma}) \times j_B;P} (M \af \toof{\Gamma}{A}{FB}\af N) \af\toof{\Gamma}{B}{\ulc} \af P
    & = & M \af\toof{\Gamma}{A}{FB} \af (N \af \toof{\Gamma \times jA}{B}{\ulc} \af
    (\piof{\Gamma}{jA} \times jB;P))
\end{eqnarray*}
\end{itemize}
\end{definition}

Thus a computation $\Gamma \vdashc M : \ulb$ is interpreted as an $\cato$-morphism $\seman{M} \ccolc \seman{\Gamma} \rightarrow \seman{\ulb}$ via
\begin{eqnarray*}
    \seman{\ttproduce V} & = & \retof{\seman{\Gamma}}{\seman{A}} \af \seman{V} \\
    \seman{M \ttto x \ttin N} & = & \seman{M} \af\toof{\seman{\Gamma}}{\seman{A}}{\seman{\ulb}} \af\seman{N}
\end{eqnarray*}
for $\Gamma \vdashv V :A$ and $\Gamma \vdashc M : FA$ and $\Gamma, x:A \vdashc N :\ulb$.  Intuitively, the naturality requirements on $\mathsf{ret}$ and $\mathsf{to}$ correspond to the fact that the term constructors $\mathtt{return}$ and $\mathtt{to}$ commute with substitution.  Note that, in the right-hand side of the associative law in Figure~\ref{fig:Flaws}, the term $P$ is implicitly weakened by $x$, and this is reflected in the right-hand side of the associative law in Definition~\ref{def:seqmodel}.

Given a RetSeq model $(\cato,\mathsf{ret}, \mathsf{to})$ on a cartesian base $(\catc,X,Y,j,F)$, an \emph{interpretation} of a signature $\cats$ interprets base value types in $X$ and base computation types in $Y$, so that every type and context has a denotation, and interprets value constructors and computation constructors as $\catc$-morphisms and $\cato$-morphisms, respectively.  We then obtain the denotation of all terms and substitutions by induction.  Our claim is that the three laws in Figure~\ref{fig:Flaws} are validated; this is proved in the usual way, using a substitution lemma.

To summarize, we have put the syntactic description of $\retseq{\cats}$ into categorical form.

\subsection{The Framework}

To bring together all our examples, we give a new definition generalizing both left-skew and right-skew multicategory.
\begin{definition}\label{def:biskew}
    A \emph{bi-skew multicategory} $\catc$ consists of the following data:
    \begin{itemize}
        \item A class  $\obs{\catc}$ of objects.
        \item For any object list $\lista$ and object $x$, a set of \emph{loose-loose} morphisms $\catc(\lista;x)$.  Such a morphism is written $f \ccolc \lista \rightarrow x$ or pictorially as,
        \[ \nti{f}{~\ora{a}~}{x}\] %
        \item For any nonempty object list $\listap$ and object $x$, a set of \emph{tight-loose} morphisms $\catc(\lti{\listap};x)$ and a set of \emph{loose-tight} morphisms $\catc(\rti{\listap};x)$.  Morphisms of these types are respectively denoted $f \ccolc \lti{\listap} \rightarrow x$ and $f \ccolc \rti{\listap} \rightarrow x$, or pictorially as,
        \[\ntile*{f}{\listap}{x} \quad \text{and} \quad \ntiri*{f}{\listap}{x}\]
        \item For any nonempty object list $\listap$ and object $x$, a set of \emph{tight-tight} morphisms $\catc(\lrti{\listap};x)$.  Such a morphism is written $f \ccolc \lrti{\listap} \rightarrow x$ or pictorially as,
        \[\ntilr*{f}{\listap}{x}\]
        \item \hiitem{Left-loosening} An operation sending
        \[\ntile*{f}{\listap}{x} ~\mapsto \nti*{\lloo{f}}{~\listap~}{x}
        \quad \text{ and } \quad
        \ntilr*{f}{\listap}{x} ~\mapsto \ntiri*{\lloo{f}}{\listap}{x}.\]
        \item \hiitem{Right-loosening} An operation
        \[\ntiri*{f}{\listap}{x} ~\mapsto \nti*{\rloo{f}}{~\listap~}{x}
        \quad \text{ and } \quad
        \ntilr*{f}{\listap}{x} ~\mapsto \ntile*{\rloo{f}}{\listap}{x}.\]
        \item \hiitem{Tight-tight identities} For each $\catc$-object $x$, $\id \ccolc \lrti{x} \to x$.
        \item \hiitem{Composition} Considering the forms in Table \ref{tab:biskew}, the following conditions must be satisfied.
        \begin{itemize}
        \item Left-loosening and right-loosening commute, i.e.\ for any $f \ccolc \lrti{\listap} \to x$, the loose-loose morphisms $\rloo{\lloo{f}}$ and $\lloo{\rloo{f}}$ are equal---we write either as  $\lrloo{f}$.
        \item Left-loosening commutes with {\rm B, D, E, F, I}.
        \item Right-loosening commutes with {\rm C, D, F, H, I}.
           \item Left identity for each kind of composition.
        \item Right identity for each kind of morphism.
         \item Associativity laws:
         \begin{math}
             \as{E}{E}, \as{G}{G}, \as{E}{F}, \as{G}{H}, \as{F}{I}, \as{H}{I}, \as{I}{I}
         \end{math}
         \item Interchange laws:
         \begin{math}
             \ch{A}{A}, \ch{E}{B}, \ch{B}{B}, \ch{C}{G}, \ch{C}{C}, \ch{F}{H}, \ch{F}{D}, \ch{D}{H}, \ch{D}{D}
         \end{math}
        \end{itemize}
        \end{itemize}
\begin{table}
    \centering
    \caption{Composition forms in a bi-skew multicategory. The labels are used in the associativity and interchange laws in Definition~\ref{def:biskew}.}
    \begin{tabular}{ll|ll|ll|ll|ll}
        A 
        & $\nti{g}{\lista,\nti{f}{~\listb~}{x},\listc}{y}$ & 
        B 
        & $\ntilesans{g}{\listap,\nti{f}{~\listb~}{x},\listc}{y}$ & 
        C 
        & $\ntirisans{g}{\lista,\nti{f}{~\listb~}{x},\listcp}{y}$ &
        D 
        & $\ntilr{g}{\listap,\nti{f}{~\listb~}{x},\listcp}{y}$ & 
        E 
        & $\ntilesans{g}[\ntile{f}{\listap}{x}]{,~\listb}{y}$ \\
        \midrule
        F 
        & $\ntilr{g}[\ntile{f}{\listap}{x}]{,~\listbp}{y}$ &
        G 
        & $\ntirisans{g}{\lista,~}[\ntiri{f}{\listb}{x}]{y}$ & 
        H 
        &  $\ntilr{g}{\listap,~}[\ntiri{f}{\listb}{x}]{y}$ & 
        I 
        & $\ntilrunary{g}{\ntilr{f}{\listap}{\lrti{x}}}{y}$
    \end{tabular}
    \label{tab:biskew}
\end{table}
\end{definition}
We assemble bi-skew multicategories into a 2-category.  Given bi-skew multicategories $\catc$ and $\catd$, a \emph{map} $F \ccolc \catc \rightarrow \catd$ sends objects to objects, and (for each kind) morphisms to morphisms, in the obvious way.  Given maps  $F,G \ccolc \catc \rightrightarrows \catd$, a \emph{2-cell} $\alpha \ccolc F \Rightarrow G$ sends each $\catc$-object $x$ to a tight-tight map $\alpha_x \ccolc \lrti{Fx} \to Gx$, satisfying the naturality condition for $\catc$-morphisms of each kind. 

Now we come to the main definition of the paper.
\begin{definition}
\label{def:biskewmonoid}
    Let $\catc$ be a bi-skew multicategory.  
    \begin{enumerate}
        \item A \emph{monoid} in $\catc$ consists of a $\catc$-object $x$, a  map $m \ccolc \lrti{x,x} \to x$ (the \emph{multiplication}), and a nullary map $e \ccolc ~~\to x$ (the \emph{unit}), satisfying the following equations:
    \[ \ntiri*{\lloo{m}}{\nti{e}{\nsp}{x},x}{x} = \ntiri*[2]{\lloo{\id}}{x}{x}
    \longquad 
    \ntile*{\rloo{m}}{x, \nti{e}{\nsp}{x}}{x} = \ntile*[2]{\rloo{\id}}{x}{x}
    \longquad
    \ntilr*{m}[\ntile{\rloo{m}}{x,x}{x}]{, x}{x} ~= \ntilr*{m}{x,}[\ntiri{\lloo{m}}{x,x}{x}]{x} \]
    \item  For monoids $M = (x,m,e)$ and $N = (y,n,u)$, a  \emph{homomorphism} $M \rightarrow N$ is a map $f \ccolc \lrti{x} \to y$ satisfying
    \[
        \nti*{\lrloo{f}}{\nti{e}{\nsp}{\,x\,}}{y} ~= \nti*[2]{u}{\nsp}{\,y\,} \longquad
        \ntilrunary*{f}{\ntilr{m}{x,x}{\lrti{x}}}{y} ~= \ntilr*{n}[\ntile{\rloo{f}}{x}{y}]{\nsp}[\ntiri{\lloo{f}}{x}{y}]{y}
    \]%
  \item We write $\monsof{\catc}$ for the category of all  monoids in $\catc$ and homomorphisms.
    \end{enumerate}
\end{definition}

As in Definition~\ref{def:leftskewify} and Proposition~\ref{prop:leftskewify}, we must see that this generalizes our previous notions:
\begin{definition}
Any left-skew multicategory $\catc$ gives a bi-skew multicategory $\biskewify (\catc)$ with the same objects.  Both the loose-loose and loose-tight hom-sets are just loose $\catc$-hom-sets, and right-loosening is identity.  Likewise, the tight-loose and tight-tight hom-sets are just tight $\catc$-hom-sets, and left-loosening is loosening.  Composition and identities are inherited from $\catc$. We extend $\biskewify$ to a 2-functor in the obvious way.
\end{definition}

\begin{proposition}
\label{prop:biskewify}
For a left-skew multicategory $\catc$, $\monsof{\catc}$ is isomorphic to $\monsof{\biskewify (\catc)}$.
\end{proposition}

Combining this result with Proposition \ref{prop:leftskewify}, it follows that the notion of monoid in a bi-skew multicategory subsumes the previously presented notions.

\subsection{Example: Category on a Class Span}

As a gentle example, recall from Section~\ref{sect:mathintro} the notion of a category on a class span. How can such a category be seen as a monoid?  We answer this as follows.
\begin{theorem}
A category on a class span $(E,A,B,f,g)$ is precisely a monoid in the following bi-skew multicategory.
\begin{itemize}
    \item An object is a set-valued matrix $X \ccolc A \relarrow B$.
 \item A morphism $X_0, \ldots, X_{n-1} \rightarrow Y$ consists of a function
 \begin{displaymath}
     X_0(fi_0,gi_{1}) \etimes \cdots \etimes X_{n-1}(fi_{n-1}, gi_{n}) \erto Y(fi_0,gi_{n})
 \end{displaymath}
 for all $i_0,\ldots, i_n \smin E$.
 \item A morphism $\lti{W,X_0, \ldots, X_{n-1}} \rightarrow Y$ consists of a function
 \begin{displaymath}
   W(a,gi_0) \etimes X_0(f i_0 ,g i_{1}) \etimes \cdots \etimes X_{n-1}(fi_{n-1}, gi_{n}) \erto Y(a,gi_{n})
 \end{displaymath}
 for all $a \smin A$ and $i_0,\ldots, i_n \smin E$.
 \item A morphism $\rti{X_0, \ldots, X_{n-1},V} \rightarrow Y$ consists of a function
 \begin{displaymath}
     X_0(fi_0,gi_{1}) \etimes \cdots \times X_{n-1}(fi_{n-1}, gi_{n}) \etimes V(fi_n, b) \erto Y(fi_0,b)
 \end{displaymath}
 for all $i_0,\ldots, i_n \smin E$ and $b \smin B$.
 \item A morphism $\lrti{W,X_0, \ldots, X_{n-1},V} \rightarrow Y$ consists of a function
 \begin{displaymath}
   W(a,gi_0) \etimes X_0(f i_0 ,g i_{1}) \etimes \cdots \etimes X_{n-1}(fi_{n-1}, gi_{n}) \etimes V(fi_n,b) \erto Y(a,b)
 \end{displaymath}
 for all $a \smin A$ and $i_0,\ldots, i_n \smin E$ and $b \smin B$.
 \item A morphism $\lrti{W} \rightarrow Y$ consists of a function
 \begin{displaymath}
     W(a,b) \erto Y(a,b)
 \end{displaymath}
 for all $a \smin A$ and $b \smin B$.
 \item Left-loosening, right-loosening, composition and identities are defined in the evident way.
\end{itemize}
\end{theorem}

\subsection{Example: Modelling of CBPV Sequencing}

Recall from Section~\ref{ssec:CBPV} that a RetSeq model on a cartesian base is used to interpret returning and sequencing in CBPV.  How can this structure be seen as a monoid?

To answer this question, we consider representable presheaves.  For a category $\catc$, the hyperlarge category $[\catc^{\op},\Set]$ is called the category of \emph{left $\catc$-modules}, or \emph{presheaves}.  Any $\catc$-object $a$ gives a presheaf $\catc(-,a)$, which is called \emph{representable}\footnote{Beware that this is distinct from the notion of representability for multicategories introduced in Definitions \ref{def:representable} and \ref{def:lerep}.}.  In the case that $\catc$ is cartesian, we can always exponentiate by a representable presheaf.
\begin{proposition} \label{prop:longshort}
    Let $\catc$ be a cartesian category.  For any left $\catc$-module $H$ and $\catc$-object $a$, the exponential of $H$ by $\catc(-,a)$ is $H(- \times a)$.  This means that we have a bijective correspondence, natural in the left module $G$, between 
\begin{itemize}
    \item natural transformations from $\catc(-,a) \times G$ to $H$, which we call \enquote{long-form}
    \item natural transformations from $G$ to $H(- \times a)$, which we call \enquote{short-form}.
\end{itemize}
Explicitly: a long-form $\alpha$ corresponds to the short-form that, at a $\catc$-object $x$, sends $r  \ccolc x \rightarrow \cdot$ to $\alpha_{x \times a}(\pi'_{x,a},\pi_{x,a};r)$.  Conversely, a short-form $\beta$ corresponds to the long-form that, at $\catc$-object $y$, sends $f \ccolc y \rightarrow a$ and $s \ccolc y \rightarrow \cdot$ to $\tuple{\id_y,f};(\beta_y(s))$.
\end{proposition}
\begin{proof}
    By direct calculation.
\end{proof}

\begin{theorem}
On a cartesian base $(\catc,X,Y,j,F)$, a  RetSeq model is precisely a monoid in the following bi-skew multicategory.  
\begin{itemize}
    \item An object is a bimodule from $\catc$ to $Y$.
   \item A morphism $\cato_0, \ldots, \cato_{n-1} \rightarrow \Omega$ is a choice of functions,
 \begin{displaymath}
  \cato_0(\Gamma \times jA_0, FA_1) \times \cdots  \times \cato_{n-1}(\Gamma \times jA_{n-1}, FA_{n}) \erto \Omega(\Gamma \times jA_{0}, FA_n)
 \end{displaymath}
 natural in $\Gamma \in \catc$, for all $A_0,\ldots,A_{n} \smin X$.  
 \item A morphism $\lti{\Phi,\cato_0, \ldots, \cato_{n-1}} \rightarrow \Omega$ is a choice of functions,
\begin{displaymath}
 \Phi(\Gamma, FA_0) \etimes 
 \cato_0(\Gamma \times jA_0, FA_1) \times \cdots \times \cato_{n-1}(\Gamma \times jA_{n-1}, FA_{n}) \erto \Omega(\Gamma, FA_n)
 \end{displaymath}
 natural in $\Gamma \in \catc$, for all $A_0,\ldots,A_{n} \smin X$.
 \item A morphism $\rti{\cato_0, \ldots, \cato_{n-1},\Psi} \rightarrow \Omega$ is a choice of functions,
 \begin{displaymath}
\cato_0(\Gamma \times jA_0, FA_1) \times \cdots \times \cato_{n-1}(\Gamma \times jA_{n-1}, FA_{n})
\times \Psi(\Gamma \times jA_n, \ulb)
\erto \Omega(\Gamma \times jA_{0}, \ulb)
 \end{displaymath}
 natural in $\Gamma \in \catc$, for all $A_0,\ldots,A_{n} \smin X$ and $\ulb \smin Y$.
 \item A morphism $\lrti{\Phi,\cato_0, \ldots, \cato_{n-1},\Psi} \rightarrow \Omega$ is a choice of functions,
 \begin{displaymath}
\Phi(\Gamma, FA_0) \etimes
 \cato_0(\Gamma \times jA_0, FA_1) \times \cdots  \times \cato_{n-1}(\Gamma \times jA_{n-1}, FA_{n})
 \times \Psi(\Gamma \times jA_n, \ulb) \erto \ \Omega(\Gamma, \ulb)
 \end{displaymath}
 natural in $\Gamma \in \catc$, for all $A_0,\ldots,A_{n} \smin X$ and $\ulb \smin Y$.
 \item A morphism $\lrti{\Phi} \rightarrow \Omega$ is a choice of functions,
 \begin{displaymath}
\Phi(\Gamma, \ulb) \ \erto \ \Omega(\Gamma,\ulb)
 \end{displaymath}
  natural in $\Gamma \in \catc$, for all $\ulb \smin Y$.
 \item Left-loosening, right-loosening, composition and identities are defined in the evident way.
\end{itemize}
\end{theorem}
\begin{proof}
    A morphism $\cato_0, \ldots, \cato_{n-1} \rightarrow \Omega$ can be expressed in long-form as a function
 \begin{displaymath}
  \catc(\Gamma, jA_0) \times
  \cato_0(\Gamma \times jA_0, FA_1) \times \cdots  \times \cato_{n-1}(\Gamma \times jA_{n-1}, FA_{n}) \erto \Omega(\Gamma, FA_n),
 \end{displaymath}
 natural in $\Gamma \in \catc$, for all $A_0,\ldots,A_{n} \smin X$.  Likewise a morphism $\rti{\cato_0, \ldots, \cato_{n-1},\Psi} \rightarrow \Omega$ can be expressed in long-form as a function
 \begin{displaymath}
\catc(\Gamma, jA_0) \times
\cato_0(\Gamma \times jA_0, FA_1) \times \cdots  \times \cato_{n-1}(\Gamma \times jA_{n-1}, FA_{n})
\times \Psi(\Gamma \times jA_n, \ulb) \erto \   \Omega(\Gamma, \ulb)
 \end{displaymath}
 natural in $\Gamma \in \catc$, for all $A_0,\ldots,A_{n} \smin X$ and $\ulb \smin Y$.  As such, we use the construction of Proposition~\ref{prop:longshort} to convert Definition~\ref{def:seqmodel} as follows.  A \emph{short-form RetSeq model} consists of the following data.
    \begin{itemize}
        \item A bimodule $\cato$ from $\catc$ to $Y$.
    \item For each $A \smin X$, a function
$\shortretof{\Gamma}{A} \ccolc 1 \rightarrow \catc(\Gamma \times jA, FA)$ natural in $\Gamma$.
\item For each $A \smin X$ and $\ulb \smin Y$, a  function 
    \begin{displaymath}
    \etoof{\Gamma}{A}{\ulb} \ccolc    \cato(\Gamma, FA) \times \cato(\Gamma \times jA, \ulb) \rightarrow \cato(\Gamma, \ulb)
    \end{displaymath}
natural in $\Gamma\smin \catc$. 
    \end{itemize}
      The following laws must be satisfied.
   \begin{itemize}
       \item For any  $\cato$-morphism $M \ccolc \Gamma \times jA \rightarrow \ulb$, 
       we have
       \begin{eqnarray*}
\shortretoof{\Gamma}{A} \af \toof{\Gamma \times jA}{A}{\ulb} \af(\pi_{\Gamma,jA} \times jA);M & = & M
       \end{eqnarray*}
  \item For any $\cato$-morphism $M \ccolc \Gamma \rightarrow FA$, we have 
  \begin{eqnarray*}
      M \af \toof{\Gamma}{A}{FA}\af \shortretoof{\Gamma}{A}   & = & M
  \end{eqnarray*}
  \item Associativity as in Definition~\ref{def:seqmodel}.  
\end{itemize}
This is easily seen to be a monoid in the stated bi-skew multicategory.
\end{proof}




\section{Unbiased Monoids}
\label{sec:unbiased}

We now give an alternative definition of monoid, which explicitly includes $n$-ary composition for every $n \smin \nats$, not just for $n=0$ and $n=2$.  The term \enquote{unbiased} is taken from~\cite{Leinster:highopshighcats}.
\begin{definition}
\label{def:unbiased}
For a multicategory (resp. left-skew, right-skew or bi-skew multicategory) $\catc$, the \emph{category of unbiased monoids}, written $\unbmon{\catc}$, is the hom-category (in the respective 2-category of multicategories) from the terminal object $1$ to $\catc$.
\end{definition}
Let us unpack this definition.  For $n \smin \nats$ and an object $x$, we write $\xn$ for the list consisting of $x$ repeated $n$ times.
\begin{itemize}
    \item In a multicategory $\catc$, an unbiased monoid consists of the following data. 
    \begin{itemize}
        \item A $\catc$-object $x$.
        \item For each $n \smin \nats$, a $\catc$-morphism $\multof{n} \ccolc \xn\rightarrow x$, known as \emph{$n$-ary multiplication}. 
       \end{itemize} 
       These must be compatible with identities (i.e.\ $m_1 = \id_x$) and composition.  The \emph{biased counterpart} is the monoid $(x,\multof{2},\multof{0})$.  Lastly, a  \emph{homomorphism} from $(x, \ldots)$ to $(y, \ldots)$ is a $\catc$-morphism $x \rightarrow y$ compatible with each kind of multiplication.
    \item  In a left-skew multicategory $\catc$, an unbiased monoid consists of the following data.
    \begin{itemize}
          \item A $\catc$-object $x$.
        \item For each $n \smin \nats$, a $\catc$-morphism $\lmultof{n} \ccolc \xn\rightarrow x$.
    \item For each $n > 0$, a $\catc$-morphism $\tmultof{n} \ccolc \lti{\xn} \rightarrow x$.
       \end{itemize} 
       These must be compatible with loosening, identities, and each kind of composition.   The \emph{biased counterpart} is the monoid $(x,\tmultof{2},\lmultof{0})$.  Lastly, a  \emph{homomorphism} from $(x, \ldots)$ to $(y, \ldots)$ is a $\catc$-morphism $\lti{x}  \rightarrow y$ compatible with each kind of multiplication.
    \item In a bi-skew multicategory $\catc$, an unbiased monoid consists of the following data.
    \begin{itemize}
          \item A $\catc$-object $x$.
        \item For each $n \smin \nats$, a $\catc$-morphism $\llmultof{n} \ccolc \xn\rightarrow x$.
    \item For each $n > 0$, $\catc$-morphisms $\tlmultof{n} \ccolc \lti{\xn} \rightarrow x$, $\ltmultof{n} \ccolc \rti{\xn} \rightarrow x$ and $\ttmultof{n} \ccolc \lrti{\xn} \rightarrow x$.
       \end{itemize} 
       These must be compatible with loosening, identities, and each kind of composition.   The \emph{biased counterpart} is the monoid $(x,\ttmultof{2},\llmultof{0})$.  Lastly, a \emph{homomorphism} from $(x, \ldots)$ to $(y, \ldots)$ is a $\catc$-morphism $\lrti{x}  \rightarrow y$ compatible with each kind of multiplication.  
\end{itemize}
(In the last case, the reason the biased counterpart is a monoid is that both sides of the left identity law are $\ltmultof{1}$, both sides of the right identity $\tlmultof{1}$ and both sides of the associativity law are $\ttmultof{3}$.  Similarly for the other cases.)

\begin{theorem}[Coherence for monoids]
\label{thm:unbiased}
Let $\catc$ be a multicategory, left-skew multicategory, right-skew multicategory or bi-skew multicategory.  Then we have an isomorphism $\unbmon{\catc} \cong \monsof{\catc}$ sending an unbiased monoid to its biased counterpart and a homomorphism to itself.
\end{theorem}
\begin{proof}
We sketch a proof for the bi-skew case: the other cases follow, including the well-known multicategory case.

We firstly show that a monoid $(x,m,e)$ is the biased counterpart of a unique unbiased monoid (dubbed the \enquote{unbiased counterpart}).  To this end, we define the data of the unbiased monoid inductively as follows (for each $n \geq 0$):
\begin{align*}
    \llmultof{0} &\eqdef e & &&
    \llmultof{n+1} &\eqdef \nti*{\lrloo{m}}{x}[\nti{\llmultof{n}}{~x,\dotsc,x~}{x}]{x} 
    &
    \tlmultof{n+1} &\eqdef \ntile*{\rloo{m}}{x}[\nti{\llmultof{n}}{~x,\dotsc,x~}{x}]{x}\\
    \ltmultof{1} &\eqdef \ntiri*{\lloo{\id}}{x}{x}&
    \ttmultof{1} &\eqdef \ntilr*{\id}{x}{x}&
    \ltmultof{n+2} &\eqdef \ntiri*{\lloo{m}}{x}[\ntiri{\ltmultof{n+1}}{x,\dotsc,x}{x}]{x}&
    \ttmultof{n+2} &\eqdef \ntilr*{m}{x}[\ntiri{\ltmultof{n+1}}{x,\dotsc,x}{x}]{x}
\end{align*}
By induction, the requirements on unbiased monoids makes each of these the only possible choice.  Induction also shows that loosening is respected, in the sense that the loosening of $\tlmultof{n+1}$ is exactly $\llmultof{n+1}$, and similarly for the other three possibilities. 
To complete the proof, we must show that these morphisms are closed under composition.  In the first instance, consider a composite:
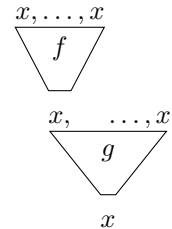
\[\nti*{g}[\nti{f}{x,\dotsc,x}{x,}]{\dotsc,x}{x}\]
where $f,g$ vary amongst the morphisms defined above. We perform induction on the arity of $f$.
\begin{itemize}
    \item Arity 0: $f = \llmultof{0}$, $g = \llmultof{n+1}$ or $\ltmultof{n+1}$:
    \[\nti*{\lrloo{m}}{\nti{e}{\nsp}{x}~}[\nti{\llmultof{n}}{~x,\dotsc,x~}{x}]{x}
    \quad \quad \quad
    \ntiri*{\lloo{m}}{\nti{e}{\nsp}{x}~}[\ntiri{\ltmultof{n}}{~x,\dotsc,x}{x}]{x}\]
    The left unit law allows us to reduce these to $\llmultof{n}$ and $\ltmultof{n}$, respectively.
    \item Arity 1 (base cases): $f = \ltmultof{1}$ and $g = \ltmultof{1}$ or $f = \ttmultof{1}$ and $g = \ttmultof{1}$. These $f$ are identities, so the composite is just $g$. For $g$ of higher arity, $\ltmultof{1}$ and $\ttmultof{1}$ cannot be composed into the first input without loosening, so these reduce to the previous cases.
    \item Inductive step (loose case): $f = \llmultof{n+1}$ and $g = \llmultof{k+1}$ or $\ltmultof{k+1}$. Expanding the definitions:
    \[\nti*{\lrloo{m}}[
    \nti{\lrloo{m}}{x}[
    \nti{\llmultof{n}}{~x,\dotsc,x~}{x}]{x}]
    [\nsp x \nsp]{~}[
    \nti{\llmultof{k}}{~x,\dotsc,x~}{x}]
    {x}
    \quad \quad \quad
    \ntiri*{\lloo{m}}[%
    \nti{\lrloo{m}}{x}[%
    \nti{\llmultof{n}}{~x,\dotsc,x~}{x}]{x}]%
    [\nsp x \nsp]{~}[%
    \ntiri{\ltmultof{k}}{~x,\dotsc,x}{x}]%
    {x}\]
    In each case we apply (loosened) associativity to the bottom and left morphisms to produce:
    \[\nti*{\lrloo{m}}{x}[
    \nti{\lrloo{m}}[
    \nti{\llmultof{n}}{~x,\dotsc,x~}{x}]{~}[
    \nti{\llmultof{k}}{~x,\dotsc,x~}{x}]{x}]
    {x}
    \quad \quad \quad
    \ntiri*{\lloo{m}}{x}[%
    \ntiri{\lloo{m}}[%
    \nti{\llmultof{n}}{~x,\dotsc,x~}{x}]{~}[%
    \ntiri{\ltmultof{k}}{~x,\dotsc,x}{x}]{x}]%
    {x}\]
    By the induction hypotheses, the upper composites are $\llmultof{k+n}$ and $\ltmultof{k+n}$, respectively, whence the overall composites are $\llmultof{k+n+1}$ and $\ltmultof{k+n+1}$ by definition of the latter.
    \item Inductive step (tight case): the arguments for $f = \tlmultof{n+1}$ are analogous.
\end{itemize}
To complete the proof, we use the above as a base case for the situation where $f$ is composed into an input $i$ further to the right. There are two situations to consider:
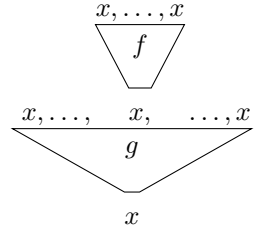
\[\nti*{g}[x,\dotsc,][\nsp x,\dotsc,]{\nti{f}{x,\dotsc,x}{x,}}[\dotsc,x][\dotsc,x \nsp]{x}
\qquad \text{and} \qquad
\nti*{g}[x,\dotsc,][\nsp x,\dotsc,]{}[\nti{f}{x,\dotsc,x}{x}][x~]{x}\]
If the index is not the right-most, then $f$ must be $\llmultof{n}$ for some $n$. Observe that the inductive construction is such that we may decompose $g = \ttmultof{k+1}$ as follows:
    \[\ntilr*{\ttmultof{i+1}}{x,\dotsc,}[%
    \ntiri{\ltmultof{k-i+1}}[%
    \nti{\llmultof{n}}{~x,\dotsc,x~}{x}]%
    {,\dotsc,x}{x} \nsp \nsp][\nsp x \nsp]%
    {x}\]
whence the top factor reduces by the previous case, becoming $\ltmultof{k+n-i}$ and the overall composite is thus $\ltmultof{k+n+1}$. The remaining possibilities for $g$ are computed similarly.

Finally, when $f$ is in the right-most argument, the composite is of the required form by the reverse of the argument needed to decompose $g$ above.

It remains to show that a monoid homomorphism is also a homomorphism between the unbiased counterparts, but this is immediate from the way the unbiased counterpart is constructed.
\end{proof}


% ...

\section{Conclusion and Further Work}

Progressing through three levels of generality, we have seen that many different notions turn out to be instances of monoids in the suitably expanded context of left-skew or bi-skew multicategories, including that of relative monad, category on a class span, and CBPV RetSeq model (which puts the syntactic description of $\retseq{\cats}$ into categorical form). The key to achieving this was a formalism in which left and right factors are given a special status, realised as a distinction between tight and loose morphisms.

In light of the related work cited throughout, we cannot claim that the resulting notion of monoid (Definition \ref{def:biskewmonoid}) is the maximally general definition that happens to be backwards-compatible with established definitions of monoid.  The kind of multicategory considered in~\cite[Remark 4.5]{Arkor-McDermotupdate}, which one might call `jumbo left-skew', has morphisms whose domain is an arbitrary monoidal expression in the objects, such as $I\otimes (x \otimes (I\otimes I))\otimes (y\otimes z)$, and a single object as codomain. 
On the one hand, our notions of monoid and homomorphism in a left-skew multicategory subsume the corresponding notions in a jumbo left-skew multicategory. On the other hand, the extra structure available in a jumbo left-skew multicategory gives rise to a richer notion of unbiased monoid, and therefore to a coherence question that  Theorem~\ref{thm:unbiased} above cannot answer.  We leave this issue to further work.

While we know of no natural example of a monoid notion that our framework excludes, one could ask if there are further possible applications in programming language theory, or whether a more syntactic treatment of these structures is possible.

We have mentioned Hermida's theorem on representable multicategories~\cite[\S9.2]{Hermida}, and Bourke and Lack's on left-representable left-skew multicategories~\cite[Theorem 6.1]{BL}. Is there an analogous result for bi-skew multicategories?  We do not see one, but further research on this question may be possible.

More generally, considering Bourke and Lack definition of $\catt$-multicategory leads naturally to the question of which structure on a $\Cat$-operad $\catt$ allows one to define a notion of (biased) `monoid in a $\catt$-multicategory', and under what further conditions a coherence theorem in the style of Theorem \ref{thm:unbiased}, relating these to the corresponding notion of unbiased monoid, would hold.

Lastly, we have yet to consider the question, ``What is a monoid action?'' An answer would allow algebras of a relative monad to be treated in a common framework with the algebras for ordinary monads, amongst other examples.  This would have  applications in programming language theory, such as to the syntax with binding of Ahrens \cite{Ahrens}. We expect the appropriate setting to be a suitable notion of multi\textit{act}egory, similarly comparable to the one developed by Arkor and McDermott \cite{Arkor-McDermotupdate}.

\newpage

\bibliographystyle{ACM-Reference-Format}
\bibliography{bibliography}

\end{document}